\documentclass[10pt,reqno]{amsart}
\usepackage{amsmath,amsthm,amsfonts,amscd,amssymb,latexsym,mathrsfs}
\usepackage{comment}
\usepackage{stmaryrd}
\usepackage{graphicx}
\usepackage[pdfusetitle,colorlinks]{hyperref}
\PassOptionsToPackage{dvipsnames}{xcolor}
\hypersetup{citecolor=OliveGreen,linkcolor=Mahogany,urlcolor=Plum}
\usepackage[utf8]{inputenc}
\usepackage[capitalise]{cleveref}
\usepackage{enumerate}
\usepackage[all,cmtip]{xy}
\usepackage{tikz}
\usepackage{float}
\usepackage{tikz-cd}
\usetikzlibrary{arrows}
\usetikzlibrary{patterns}
\usetikzlibrary{decorations.markings}

\usepackage[%
    style=alphabetic,sorting=nyt,
    maxnames=10,
    sortcites=true,doi=false,
    giveninits=true,hyperref,backend=bibtex]{biblatex}
    \renewbibmacro{in:}{%
      \ifentrytype{article}{}{\printtext{\bibstring{in}\intitlepunct}}}
\DeclareFieldFormat
  [article,inbook,incollection,inproceedings,patent,thesis,unpublished]
  {title}{{#1\isdot}}
\DeclareFieldFormat{pages}{#1}

\usepackage{cancel}
\AtEveryBibitem{%
\ifentrytype{book}{
    \clearfield{url}%
    \clearfield{pages}
    \clearfield{isbn}
    \clearfield{urldate}%
    \clearfield{review}%
    \clearfield{series}
}{}
\ifentrytype{article}{
    \clearfield{url}%
    \clearfield{issn}
    \clearfield{doi}
    \clearfield{isbn}
    \clearfield{urldate}%
    \clearfield{review}%
}{}
\ifentrytype{collection}{
    \clearfield{url}%
    \clearfield{urldate}%
    \clearfield{review}%
    \clearfield{series}
}{}
\ifentrytype{incollection}{
    \clearfield{url}%
    \clearfield{urldate}%
    \clearfield{review}%
    \clearfield{series}
}{}
}

\usepackage{silence,lmodern}
\numberwithin{equation}{subsection}

\newtheoremstyle{plain2}    
   {}            
   {}            
   {\itshape}    
   {}            
   {\bfseries}   
   {.}           
   {5pt plus 1pt minus 1pt}  
   {{\thmnumber{#1} \thmname{#2}{\thmnote{ (#3)}}}}          

\newtheorem{othercustomthm}{Theorem}
\newenvironment{customthmtwo}[1]
  {\renewcommand\theothercustomthm{#1}\othercustomthm}
  {\endothercustomthm}
  
\theoremstyle{plain2}

\newtheorem{thmy}{Theorem}

\newtheorem{theorem}[equation]{Theorem}
\newtheorem{corollary}[equation]{Corollary}
\newtheorem{lemma}[equation]{Lemma}
\newtheorem{proposition}[equation]{Proposition}
\newtheorem{conjecture}[equation]{Conjecture}
\newtheorem{cor}[equation]{Corollary}
\newtheorem{prop}[equation]{Proposition}

\newtheorem{Assumption}[equation]{Assumption}

\newtheorem{definition}[equation]{Definition}
\theoremstyle{definition}

\newtheorem{remark}[equation]{Remark}

\newtheoremstyle{stepstyle}
   {}     {}
   {\normalfont}
   {\parindent}
   {\itshape}
   {}
   {5pt plus 1pt minus 1pt}
   {{\thmname{#1} \thmnumber{#2}:{\thmnote{#3}}}}
\theoremstyle{stepstyle}

\newtheoremstyle{point}
   {}     {}
   {\normalfont}
   {}
   {\bfseries}
   {}
   {5pt plus 1pt minus 1pt}
   {{\thmname{#1}(\thmnumber{#2})\thmnote{ #3.}}}
\theoremstyle{point}
\newtheorem{point}[equation]{}

\newtheoremstyle{subpoint}
   {}     {}            
   {\normalfont}
   {}                   
   {\normalfont}
   {}
   {5pt plus 1pt minus 1pt}
   {{\thmname{#1}{\bf (\thmnumber{#2})}\thmnote{ #3.}}}
\theoremstyle{subpoint}
\newtheorem{subpoint}[equation]{}

\newcommand{\spa}[1]{\begin{subpoint}#1\end{subpoint}}           

\usepackage{combelow}

\usepackage{tikz-cd}

\usepackage{enumitem}

\newcommand{\cG}{\ensuremath{\mathscr{G}}}
\newcommand{\cS}{\ensuremath{\mathscr{S}}}
\newcommand{\cV}{\ensuremath{\mathscr{V}}}

\newcommand{\cY}{\ensuremath{\mathscr{Y}}}
\newcommand{\cK}{\ensuremath{\mathscr{K}}}
\newcommand{\cX}{\ensuremath{\mathscr{X}}}
\newcommand{\cE}{\ensuremath{\mathscr{E}}}
\newcommand{\cN}{\ensuremath{\mathscr{N}}}

\newcommand{\SL}{\mathrm{SL}}

\newcommand{\Gr}{\mathrm{Gr}}
\newcommand{\Autom}{\mathrm{Aut}}
\newcommand{\St}{\mathrm{St}}
\newcommand{\Stcl}{\overline{\mathrm{St}}}
\newcommand{\Link}{\mathrm{Link}}

\DeclareMathOperator{\Hom}{Hom}

\let\ker\relax
\DeclareMathOperator{\ker}{ker}

\newcommand{\Z}{\mathbb{Z}}
\newcommand{\N}{\mathbb{N}}
\newcommand{\R}{\mathbb{R}}
\newcommand{\Q}{\mathbb{Q}}
\newcommand{\C}{\mathbb{C}}
\newcommand{\A}{\mathbb{A}}

\renewcommand{\P}{\mathbb{P}}

\DeclareMathOperator{\Spec}{Spec}

\DeclareMathOperator{\PL}{PL}

\DeclareMathOperator{\ord}{ord}

\DeclareMathOperator{\Sk}{Sk}

\DeclareMathOperator{\Skess}{Sk}
\DeclareMathOperator{\Cone}{Cone}

\newcommand{\Log}{\textrm{Log}}
\newcommand{\codim}{\textrm{codim}}

\newcommand{\toric}{Z}
\newcommand{\Diff}{\operatorname{Diff}}
\newcommand{\DMR}{\mathcal{DMR}}

\usetikzlibrary{matrix}

\usepackage{caption}
\usepackage{subcaption}

\crefname{section}{\S}{\S\S}
\Crefname{section}{\S}{\S\S}

\crefformat{equation}{(#2#1#3)}

\makeatletter
\def\@tocline#1#2#3#4#5#6#7{\relax
  \ifnum #1>\c@tocdepth 
  \else
    \par \addpenalty\@secpenalty\addvspace{#2}%
    \begingroup \hyphenpenalty\@M
    \@ifempty{#4}{%
      \@tempdima\csname r@tocindent\number#1\endcsname\relax
    }{%
      \@tempdima#4\relax
    }%
    \parindent\z@ \leftskip#3\relax \advance\leftskip\@tempdima\relax
    \rightskip\@pnumwidth plus4em \parfillskip-\@pnumwidth
    #5\leavevmode\hskip-\@tempdima
      \ifcase #1
       \or\or \hskip 1em \or \hskip 2em \else \hskip 3em \fi%
      #6\nobreak\relax
    \dotfill\hbox to\@pnumwidth{\@tocpagenum{#7}}\par
    \nobreak
    \endgroup
  \fi}

\renewcommand{\O}{\mathcal{O}}

\DeclareMathOperator{\Supp}{Supp}

\newcommand{\Aut}{\textrm{Aut}}

\newcommand{\PP}{\mathcal{P}}

    \usepackage{soul}
    
\usepackage{mathtools}

\newcommand{\caO}{\ensuremath{\mathcal{O}}}

\DeclareMathOperator{\dlt}{dlt}

\renewcommand{\PP}{\mathbb{P}}

\newcommand{\FF}{\mathfrak{f}}

\usepackage{chngcntr}
\counterwithin{figure}{subsection}

\usepackage[title]{appendix}
\newcommand{\dS}{R}
\newcommand{\cZ}{\ensuremath{\mathscr{Z}}}
\newcommand{\Gl}{\mathrm{GL}_n}
\newcommand{\Sl}{\mathrm{SL}_n}
\newcommand{\PGl}{\mathrm{PGL}_n}

\title{On the geometric P=W conjecture}

\author{Mirko Mauri}
\address{Department of Mathematics, University of Michigan, Ann Arbor, MI
48109--1043, USA}
\email{\href{mailto:mauri@mpim-bonn.mpg.de}{mirmom@umich.edu}} 
\author{Enrica Mazzon}
\address{Department of Mathematics, University of Michigan, Ann Arbor, MI
48109--1043, USA}
\email{\href{mailto:mazzon@mpim-bonn.mpg.de}{mazzon@umich.edu}}
\author{Matthew Stevenson}
\address{Department of Mathematics, University of Michigan, Ann Arbor, MI
48109--1043, USA}
\email{\href{mailto:stevmatt@umich.edu}{stevmatt@umich.edu}}

\date{\today}

\begin{document}

\begin{abstract}
We formulate the geometric P=W conjecture for singular character varieties. We establish it for compact Riemann surfaces of genus one, and obtain partial results in arbitrary genus. To this end, we employ non-Archimedean, birational and degeneration techniques to study the topology of the dual boundary complex of character varieties. We also clarify the relation between the geometric and the cohomological P=W conjectures.
\end{abstract}

\maketitle

\setcounter{tocdepth}{1}
\section{Introduction}
Given a smooth quasi-projective variety $X$, choose a compactification $\overline{X} = X \cup \partial X$ such that the boundary $\partial X$ has simple normal crossings (snc). The \emph{dual boundary complex} $\mathcal{D}(\partial X)$ is a cell complex encoding information about the irreducible components of $\partial X$ and their intersections. Roughly, the combinatorics of $\mathcal{D}(\partial X)$ tells us about the geometry of $X$ at infinity. 

The purpose of this paper is to study the topology of dual boundary complexes of certain character varieties in relation to the geometric P=W conjecture. 

\subsection{The P=W conjectures.} Let $C$ be a  Riemann surface of genus $g$, and let $G$ be a complex reductive algebraic group. 
The Dolbeault and Betti moduli spaces $M_{\mathrm{Dol}}(C,G)$ and $M_{\mathrm{B}}(C,G)$ are central objects in non-abelian Hodge theory. The former parametrises principal $G$-Higgs bundles with vanishing Chern classes, while the latter, also called $G$-character variety, parametrises representations of the fundamental group of $C$ with values in $G$; see \cref{subsec:defnMBMDol}. Although the algebraic moduli problems which they solve are very
different, there exists a canonical real analytic isomorphism 
\[\Psi\colon M_{\mathrm{Dol}}(C,G) \to M_{\mathrm{B}}(C,G),\]
 called \emph{non-abelian Hodge correspondence}. Note that $\Psi$ is not a biholomorphism: the Dolbeault moduli space carries a projective  Lagrangian fibration onto an affine space
\[\chi: M_{\mathrm{Dol}}(C,G) \to \C^{\frac{1}{2}\dim M_{\mathrm{Dol}}(C,G)},\]
called \emph{Hitchin fibration}, whose regular fibres are holomorphic Lagrangian tori; instead the Betti moduli space does not contain any proper subvariety of positive dimension, since it is affine. 

The geometric P=W conjecture aims to a better understanding of the asymptotic behaviour at infinity of the transcendental map $\Psi$.
Formulated by Katzarkov, Noll, Pandit, and Simpson in~\cite[Conjecture 1.1]{KatzarkovNollPanditEtAl2015}, it says in brief that
\vspace{7pt}

\noindent \textbf{Geometric P=W conjecture} (abridged version)\textbf{.}
\emph{The dual boundary complex $\mathcal{D}(\partial M_{\mathrm{B}}(C,G))$ is homotopy equivalent to a sphere.}
\vspace{7pt}

\noindent The dual boundary complex should furthermore be naturally identified with the sphere at infinity on the base of the Hitchin fibration. 
The full statement actually involves the commutativity of a square diagram, but we postpone the precise formulation of the conjecture to \cref{sec:geometricP=W}. There we propose a tweaked version of the original statement in \cite{KatzarkovNollPanditEtAl2015} necessary to deal with the singularities of the Betti and Dolbeault moduli spaces.

The geometric P=W conjecture takes its name after a conjecture formulated by de Cataldo, Hausel and Migliorini in \cite{CataldoHauselMigliorini2012}, called cohomological P=W conjecture; see \cref{sec:cohomologicalP=Wconjecture}. The different geometry of $M_{\mathrm{Dol}}(C,G)$ and $M_{\mathrm{B}}(C,G)$ is encoded in two filtrations on their isomorphic singular cohomology, named respectively P and W, which stand for perverse and weight filtrations. 
Despite the different origin of the filtrations, the conjecture asserts that
\vspace{7pt}

\noindent \textbf{Cohomological P=W conjecture.}
\emph{The non-abelian Hodge correspondence $\Psi$ exchanges the perverse and weight filtrations.}
\vspace{7pt}

In \cref{sec:p=wgeomcohom}, we clarify the relation between the cohomological and geometric conjectures, thus answering the final question of \cite{Migliorini}; see also \cite[\S 1]{Harder2019}\footnote{Note that in \cref{introthm:cohomologicalvsgeometric} we do not need to assume the existence of an snc logCY compactification of $M_{\mathrm{B}}(C,G)$ as in loc. cit.}. 

\begin{customthmtwo}{A}
[\cref{thm:cohomologicalvsgeometric}] \label{introthm:cohomologicalvsgeometric}
Under the technical \cref{assump:genericsmoothness}, the geometric P=W conjecture implies the cohomological P=W conjecture at the highest weight.
\end{customthmtwo}

\subsection{New evidence for the geometric P=W conjecture.}
Up to date, the evidence for the geometric P=W conjecture in the literature concerns only the case of punctured spheres, and the commutativity of the square diagram is known only in few low-dimensional examples; see \cref{rmk:relationotherworksIII}. 
In this paper, following a suggestion of Simpson in ~\cite{Simpson2016}, we study the case of Betti moduli spaces associated to compact Riemann surfaces. More precisely, we establish the geometric P=W conjecture both in rank one for a compact curve of arbitrary genus, and in arbitrary rank and genus one.

\begin{customthmtwo}{B}
[\cref{thm:P=Wconjecturegeomrankonetext}, \cref{thm:P=Wconjecturegeomgenusonetext}] \label{introthm:P=Wconjecturegeomgenusonetext}
Let $C$ be a compact Riemann surface of genus $g$, and $G$ be either $\mathrm{GL}_n$ or $\mathrm{SL}_n$. 
The geometric P=W conjecture holds for $M_{\mathrm{B}}(C,G)$ if either $n=1$ or $g=1$.
\end{customthmtwo}

In arbitrary genus, we determine some topological invariants of $\mathcal{D}(\partial M_{\mathrm{B}}(C,G))$, which can be regarded as partial evidence for the geometric P=W conjecture.

\begin{customthmtwo}{C}
[\cref{cor:rational homology}, \cref{thm:simplyconnectedness}]\label{thmintro:simplyconnected}
Let $C$ be a compact Riemann surface of genus $g$, and $G$ be either $\mathrm{GL}_n$ or $\mathrm{SL}_n$.
The following facts hold:
\begin{itemize}
    \item $\mathcal{D} (\partial M_{\mathrm{B}}(C,G))$ has the rational homology of a sphere of dimension equal to $\dim M_{\mathrm{B}}(C,G) -1$ if $n=2$;
    \item $\mathcal{D}(\partial M_{\mathrm{B}}(C,G))$ is simply-connected unless $\dim M_{\mathrm{B}}(C,G)=2$, i.e. unless $g=1$ and $G=\C^*, \SL_2$.
\end{itemize}
\end{customthmtwo}

\subsection{The PL-homeomorphism type of the dual boundary complex of character varieties}
In order to address the geometric P=W conjecture, one must make sense of $\mathcal{D}(\partial M_{\mathrm{B}}(C,G))$.
It is not a priori clear how one can do so, since the affine variety $M_{\mathrm{B}}(C,G)$ can be singular, hence it may not admit an snc compactification. Thus, the task is to find mildly singular compactifications to which a dual complex is still attached. 

Our solution is to consider \emph{divisorial log terminal (dlt) compactifications}; see \eqref{subsection dlt modification} for a definition. Indeed, such compactifications have well-defined dual complex, whose homotopy type is independent of 
the choice of a specific dlt compactification by \cite{deFernexKollarXu2012}. In \cref{lem:dltcomp} we show that dlt compactifications do exist for Betti moduli spaces of compact Riemann surfaces of any genus and structural group either $\mathrm{GL}_n$ or $\mathrm{SL}_n$. 

Among all possible dlt compactifications of $M_{\mathrm{B}}(C,G)$, it is convenient to restrict to special ones, more precisely to \emph{dlt log Calabi--Yau (logCY) compactifications} (if they exist!). This is an algebraic condition which rigidifies the configuration of divisors at infinity: 
the dual complex of any dlt logCY compactification identifies a distinguished PL-homeomorphism class 
in the homotopy equivalence class
of the dual complex of any dlt compactification. Not all quasi-projective varieties admit logCY compactifications, but a folk conjecture \cite{Simpson2016} says that character varieties should do,
and this is indeed proved for $\mathrm{SL}_2$-local systems on punctured surfaces
in \cite{Whang2020}. If these compactifications were also dlt, then a positive answer to \cite[Question 4]{KollarXu} would essentially imply that $\mathcal{D}(\partial M_{\mathrm{B}}(C,G))$ is homeomorphic to a sphere. 

These observations suggest the following partial strengthening of the geometric P=W conjecture.

\begin{conjecture} \label{conj:strong}
$M_{\mathrm{B}}(C,G)$ admits a dlt logCY compactification $(\overline{M}_{\mathrm{B}}(C,G), \partial M_{\mathrm{B}}(C,G))$, and the associated dual boundary complex $\mathcal{D}(\partial M_{\mathrm{B}}(C,G))$ is $\PL$-homeomorphic to a sphere\footnote{A sphere of dimension $m$ is intended to be endowed with its standard PL-structure, i.e. with a triangulation equivalent to the boundary of an $(m+1)$-standard simplex.}.
\end{conjecture}

In \cref{sec:dualcomplexgenusone} we prove this stronger conjecture for a compact Riemann surface $\mathrm{E}$ of genus one. In particular, we obtain the following theorems.

\begin{customthmtwo}{D}
[Theorem \ref{dualcomplexhilbertscheme}] \label{thm intro P=W}
 $M_{\mathrm{B}}(\mathrm{E},\mathrm{GL}_n)$ admits a dlt logCY compactification, and the dual boundary complex $\mathcal{D}(\partial M_{\mathrm{B}}(\mathrm{E},\mathrm{GL}_n))$ of such compactification is $\PL$-homeomorphic to  $\mathbb{S}^{2n-1}$.
\end{customthmtwo}

\begin{customthmtwo}{E}
[Theorem \ref{dual boundary complex SL_n}] \label{thm intro P=W slr} $M_{\mathrm{B}}(\mathrm{E},\mathrm{SL}_n)$ admits a dlt logCY compactification, and the dual boundary complex $\mathcal{D}(\partial M_{\mathrm{B}}(\mathrm{E},\mathrm{SL}_n))$ of such compactification is $\PL$-homeomorphic to  $\mathbb{S}^{2n-3}$.
\end{customthmtwo}

It is an open question whether $M_{\mathrm{B}}(C,G)$ admits a dlt logCY compactification for $g>1$. 

\subsection{Non-Archimedean approach} 
In Theorems \ref{thm intro P=W} and \ref{thm intro P=W slr} we show that dlt logCY compactifications of $M_{\mathrm{B}}(\mathrm{E},G)$ do exist via the minimal model program: this is a non-constructive existence result, and the authors are not aware of more explicit dlt logCY compactifications. However, we can determine the PL-homeomorphism type of the dual boundary complex of such compactifications by recasting the problem in terms of non-Archimedean geometry.
In this approach, boundary divisors are thought as suitable divisorial valuations, in the sense of~\eqref{divisiorialvaluation}. 
This viewpoint permits to reinterpret the dual complex $\mathcal{D}(\partial M_{\mathrm{B}}(\mathrm{E},G))$ as the level set of a suitable function inside a space of valuations, namely as the minimality locus of a log discrepancy function, called \emph{skeleton}; see \cref{definitionesssentialskeletonlogCY} and \cref{crepantbirationalpair}. This idea is compatible with the expectation of a correspondence between valuations on the Betti moduli space and directions in the Hitchin base, conjectured in~\cite[\S 1.1]{KatzarkovNollPanditEtAl2015}. 

In particular, the valuative interpretation of dual complexes allows us to prove the following fact about their topology, which is used to prove Theorems \ref{thm intro P=W} and \ref{thm intro P=W slr}, and is of independent interest.

\begin{customthmtwo}{F}
[Theorem \ref{lem:quotientlogCYskeleton}] \label{thm dual complex intro} Let $(X, \Delta_X)$ be a log canonical pair. Let $\Gamma$ be a finite subgroup of $\Autom(X, \Delta_X)$, i.e.\ the group of automorphisms $\gamma \colon X \to X$ such that $\gamma^*\Delta_X = \Delta_X$. 
Suppose also that the quotient map $q\colon X \to X/\Gamma$ is quasi-\'{e}tale, i.e.\ the action of $\Gamma$ is free in codimension $1$. Denote $\DMR(X, \Delta_X)$ the dual complex of a dlt modification of $(X, \Delta_X)$,\footnote{The notation $\DMR$ stands for ``Dual complex of a Minimal divisorial log terminal partial Resolution".} and $\Delta_{X/\Gamma} \coloneqq q_*\Delta_{X}$. Then, there exists a $\PL$-homeomorphism
\[
    \mathcal{DMR}(X/\Gamma, \Delta_{X/\Gamma}) \simeq_{\mathrm{PL}} \mathcal{DMR}(X, \Delta_{X})/\Gamma.
\]
\end{customthmtwo}

It is worth remarking that one could determine the homotopy type of $\mathcal{D}(\partial M_{\mathrm{B}}(\mathrm{E},G))$ by using the classical notion of skeleton for toroidal pairs as in \cite{Thuillier2007}. However, the definition of skeleton in \cref{sec:dualcomplexofquotient} adapts well to the more singular case of dlt compactifications, which appear in the formulation of the geometric P=W conjecture. Further, in the logCY case, it allows to establish the actual $\PL$-homeomorphism class of the dual boundary complex. 

\subsection{Hyper-K\"{a}hler degenerations on the Betti side}\label{introsec:hyperkahler}
We propose two proofs for each of Theorems \ref{thm intro P=W} and \ref{thm intro P=W slr} (see \cref{sec:dualcomplexgenusone} and \cref{alternativeproofindiscretelyvalued}-\ref{alternativeproofdiscretevalued2}). One of the proofs requires the construction of degenerations of compact hyper-K\"{a}hler manifolds. Degeneration arguments have also been used in \cite{deCataldoMaulikShen2019} to show the cohomological P=W conjecture in genus two. There, the authors exploit a standard degeneration  of compact hyper-K\"{a}hler moduli spaces of sheaves on K3 surfaces to the Dolbeault moduli space $M_{\mathrm{Dol}}(C,\Gl)$, due to \cite[\S 1.1]{DonagiEinLazarsfeld1997}. However, no analogue is known on the Betti side, and the following \cref{intro:degenerationhyperkahler} is the first existence result in this direction.

\begin{customthmtwo}{G}
[\cref{degenerationHIlbertscheme} and \ref{construct dlt compact}] \label{intro:degenerationhyperkahler}
There exist 
\begin{itemize}
    \item a good minimal dlt degeneration $\cX$ of Hilbert schemes of K3 surfaces, or of generalised Kummer varieties,
    \item an irreducible component $\Delta^{\dlt}_i$ of the reduced special fibre $\cX_0$ of $\cX$,
\end{itemize}
 such that the pair $(\Delta_i^{\dlt} ,\Diff^*_{\Delta^{\dlt}_i}(\cX_0))$
  is crepant birational to a dlt logCY compactification of $M_{\mathrm{B}}(\mathrm{E},\mathrm{GL}_n)$ or of $M_{\mathrm{B}}(\mathrm{E},\mathrm{SL}_n)$ respectively.
\end{customthmtwo}

We postpone to \cref{preliminaries}, \cref{alternativeproofindiscretelyvalued} and \cref{alternativeproofdiscretevalued2} the explanation of all the attributes of these degenerations. Here, we just observe that \cref{intro:degenerationhyperkahler} suggests a relation of the geometric P=W conjecture with the following conjecture about dual complexes of hyper-K\"{a}hler manifolds: 

\begin{conjecture} \label{intro:conjHK}
Let $\cX$ be a maximally degenerate good minimal dlt degeneration of compact hyper-K\"{a}hler manifolds. Then the dual complex of the special fibre of $\cX$ is (PL-)homeomorphic to $\P^n(\C)$.
\end{conjecture}

\noindent Indeed, assume that a dlt compactification of ${M}_{\mathrm{B}}(C, G)$ appears as an irreducible component of the central fibre $\mathscr{X}_0$ of a degeneration $\cX$ as in \cref{intro:conjHK}.
Then $\mathcal{D}(M_{\mathrm{B}}(C, G))$ is a PL-sphere if $\mathcal{D}(\mathscr{X}_0)$ is a PL-manifold; see \cref{linkdualcomplex}. 

\subsection{Structure of the paper} The present paper is organised as follows. In \cref{preliminaries}, we provide some preliminaries about Berkovich spaces, log discrepancy and dual complexes. In \cref{sec:dualcomplexofquotient}, we study the dual complex of some quotients of log canonical pairs (\cref{thm dual complex intro}) by means of non-Archimedean techniques. In \cref{sec:geometricP=W}, we formulate the geometric P=W conjecture, and confirm it in rank one (\cref{introthm:P=Wconjecturegeomgenusonetext} for $n=1$). Then we study the PL-homeomorphism type of dual complex of character varieties in genus one: we show \cref{thm intro P=W}, \ref{thm intro P=W slr} and \ref{intro:degenerationhyperkahler} in \cref{sec:dualcomplexgenusone}, \cref{alternativeproofindiscretelyvalued} and \cref{alternativeproofdiscretevalued2}. The proof of the full geometric P=W conjecture in genus one (\cref{introthm:P=Wconjecturegeomgenusonetext} for $g=1$) is contained in \cref{sec:fullP=Wgenusone}. Finally, we compare the cohomological P=W conjecture and the geometric one in \cref{sec:cohomolovsgeometric} (\cref{introthm:cohomologicalvsgeometric}), and in \cref{sec:dualcomplexarbgenus} we study topological properties of dual complex of character varieties in arbitrary genus (\cref{thmintro:simplyconnected}).
In \cref{appendixindep} we check the independence of the geometric P=W conjecture of all the auxiliary choices made in \cref{subsec:defnMBMDol}. \cref{appendixA} contains some technical results on the Tate curve required for a proof of Theorem \ref{intro:degenerationhyperkahler}.

\subsection{Acknowledgements}
We would like to thank Omid Amini, S\'{e}bastien Boucksom,  Morgan Brown, Paolo Cascini, Tommaso de Fernex, Mattias Jonsson,  Luca Migliorini, Joaquín Moraga, Johannes Nicaise, 
Szil\'{a}rd Szab\'{o} and Marco Trozzo for helpful discussions on the contents of the paper. 
We wish to thank the anonymous referees for their careful reading and useful suggestions.

Mauri and Mazzon were supported by the University of Michigan, the Max Planck Institute for Mathematics, and the Engineering and Physical Sciences Research Council [EP/L015234/1]. 
The EPSRC Centre for Doctoral Training in Geometry and Number Theory (The London School of Geometry and Number Theory), University College London.
Stevenson 
was partially supported by NSF grant DMS-1600011, and he
is grateful to Imperial College London for hosting his visit, supported by the ERC Starting Grant MOTZETA (project 306610) of the European Research Council (PI: Johannes Nicaise).

\section{Preliminaries: Berkovich spaces, log discrepancy, and dual complexes}\label{preliminaries}

\spa{Let $X$ be a normal proper algebraic variety over an algebraically closed field $k$ of characteristic zero. A \textbf{pair} $(X, \Delta)$ is the datum of $X$ together with a Weil $\Q$-divisor $\Delta$ with coefficients in $(0,1]$, called boundary, such that $K_X+ \Delta$ is $\Q$-Cartier.
Write $\Delta^{=1}= \sum_{i \in I} \Delta_i$ for the union of all irreducible components of $\Delta$ whose coefficient equals $1$.
The irreducible components of the intersection $\Delta_{i_1} \cap \ldots \cap \Delta_{i_r}$, with $\{i_1, \ldots, i_r\}\subseteq I$, 
are called \textbf{strata} of $\Delta^{=1}$.

The pair $(X, \Delta)$ is \textbf{log Calabi--Yau} (logCY) if $K_X + \Delta$ is $\Q$-linearly equivalent to zero, written
$K_X + \Delta \sim_{\Q} 0 $.}

\spa{\label{def:Berkovich}
We denote by $X^{\mathrm{bir}}$ the locus of birational points of the \textbf{Berkovich analytification} of $X$. As a set, \(X^{\mathrm{bir}}\) is the space of real valuations of the fraction field $k(X)$ of $X$ that extend the trivial valuation $v_0$ of $k$, i.e. the set of functions $v\colon k(X)^* \to \mathbb{R}$
with the properties that:
\begin{enumerate}
\item $v(f \cdot g)=v(f) + v(g)$ for all $f, g \in k(X)$;
\item $v(f+g) \geqslant \min\{v(f), v(g)\}$ for all $f, g \in k(X)$;
\item $v(h)=0$ for all $h \in k^{*}$.
\end{enumerate} 
We endow $X^{\mathrm{bir}}$ with the coarsest topology which makes 
the maps $X^{\mathrm{bir}} \to \mathbb{R}$, $ v \mapsto v(f)$, continuous for any $f \in k(X)^*$.
Note that a dominant morphism of algebraic varieties $g\colon X\to Y$ induces a surjective continuous map
$g^{\mathrm{bir}}\colon X^{\mathrm{bir}} \to Y^{\mathrm{bir}}$ by precomposition $v(\cdot) \mapsto v(\cdot \circ g)$.

By the valuative criterion of properness, any valuation $v$ admits a \textbf{center} on $X$, i.e.\ there exists a unique schematic point $c_{X}(v)\in X$ such that $v(f) \geqslant 0$ for any $f \in \caO_{X,c_{X}(v)}$ and $v(f)>0$ if and only if $f \in \mathfrak{m}_{c_{X}(v)}$.}

\spa{\label{divisiorialvaluation}Given a non-negative real number $c$, a proper birational morphism $h\colon Y \to X$ of normal varieties, and an irreducible divisor $E \subset Y$, we define the \textbf{divisorial valuation}
\[c \cdot v_{E}(f) \coloneqq c \cdot \ord_E(h^*f),\]
where $\ord_E(h^*f)$ is the order of the function $h^*f$ along $E$. The set of divisorial valutions is denoted by $X^{\mathrm{div}} \subset X^{\mathrm{bir}}$, and it is dense in $X^{\mathrm{bir}}$; see for instance \cite[\S 4.4]{JonssonMustata}.
}

\spa{Let $(X, \Delta)$ be a pair, and $c \cdot v_{E} \in X^{\mathrm{div}}$ be the divisorial valuation determined by the triple $(c, h: Y \to X, E)$.
Pick canonical divisors $K_Y$ and $K_X$ such that $h_*(K_Y) = K_X$.
The \textbf{log discrepancy} of $c \cdot v_{E}$ is the value
\begin{equation}\label{log discrepancy 2}
A_{(X,\Delta)}(c \cdot v_{E}) \coloneqq c \left( 1 + \ord_E\left(K_Y - \frac{1}{m}h^*(m(K_X + \Delta))\right) \right)
\end{equation}
for $m \in \Z_{>0}$ sufficiently divisible. It is easy to see that the log discrepancy $A_{(X,\Delta)}(c \cdot v_{E})$ depends only on $c \cdot v_{E}$, and not on the choice of $m$ or of the birational model $Y$ of $X$ where the centre of $c \cdot v_{E}$ has codimension one. 
Define the \textbf{log discrepancy function} $A_{(X,\Delta)} \colon X^{\mathrm{bir}} \to \R \cup \{\infty\}$ as the maximal lower-semicontinuous extension of $A_{(X,\Delta)} \colon X^{\mathrm{div}} \to \R$; see for instance ~\cite[\S 5]{JonssonMustata}, ~\cite{BoucksomFernexFavreEtAl2015}, ~\cite[Appendix A]{BoucksomJonssona}.

The pair $(X,\Delta)$ is \textbf{log canonical} (lc) if $A_{(X,\Delta)}(v) \geqslant 0$ for all $v \in X^{\mathrm{bir}}$. Two pairs $(X, \Delta_X)$ and $(Y, \Delta_Y)$ are \textbf{crepant birational} if $X$ and $Y$ are birational in such a way that $A_{(X, \Delta_X)}=A_{(Y, \Delta_Y)}$ as functions on $X^{\mathrm{bir}}=Y^{\mathrm{bir}}$.
}

\spa{\label{subsection dlt modification} Given an lc pair $(X, \Delta)$, an \textbf{lc centre} of the pair is the centre in $X$ of a divisorial valuation $v \in X^{\mathrm{div}}$ with $A_{(X, \Delta)}(v)=0$. 
The \textbf{snc locus} $X^{\mathrm{snc}}$ is the largest open subset in $X$ such that the pair $(X, \Delta)$ has only simple normal crossings (snc). The pair $(X, \Delta)$ is said to be \textbf{divisorial log terminal} (dlt) if none of its lc centres are contained in $X \setminus X^{\mathrm{snc}}$. In particular, the lc centres of a dlt pair are the strata of $\Delta^{=1}$. See~\cite[Definition 2.37]{KollarMori} and \cite[Theorem 4.16]{KollarKovacs} for more details.}

\spa{There are several advantages to working with dlt pairs over snc pairs.
Most notably, we
use the fact that any lc pair $(X, \Delta)$ is crepant birational to a (non-unique) dlt pair $(X^{\text{dlt}}, \Delta^{\text{dlt}})$, while the corresponding statement fails in general for snc pairs. 
This fact is a consequence of the existence of \textbf{dlt modifications}, as in~\cite[Corollary 1.36]{Kollar2013}, which asserts that there exists a proper birational morphism $g\colon X^{\text{dlt}}\to X$ with exceptional divisors $\{ E_i\}_{i \in I}$ such that
\begin{enumerate}
    \item (dlt) the pair $(X^{\text{dlt}}, \Delta^{\text{dlt}}\coloneqq g_*^{-1}\Delta + \sum_{i \in I} E_i)$ is dlt, where $g_*^{-1}\Delta$ is the strict transform of $\Delta$ via $g$;
    \item (crepant) $K_{X^{\text{dlt}}}+ \Delta^{\text{dlt}} \sim_\Q g^*(K_X+\Delta)$.
\end{enumerate}
Denote by $\Autom(X, \Delta)$ the group of automorphism $\gamma \colon X \to X$ such that $\gamma^*\Delta_X = \Delta_X$. Given a finite subgroup $\Gamma \subseteq \Autom(X, \Delta)$, dlt modifications of $(X, \Delta_X)$ can be chosen $\Gamma$-equivariant. Indeed, the argument of \cite[Corollary 1.36]{Kollar2013} works in this context too if we replace ordinary log resolutions and MMP with their $G$-equivariant analogues.
}

\spa{\label{dltmodification}It is always possible to associate a combinatorial object to a dlt pair $(X,\Delta)$:}
\begin{definition}\label{dual intersection complex definition}The \textbf{dual complex} of a dlt boundary $\Delta$
is the regular cell complex\footnote{See \cite[\S 2.1]{Hatcher} for the definition of regular cell complex or regular $\Delta$-complex.} $\mathcal{D}(\Delta)$ whose vertices are in correspondence with the irreducible components of $\Delta^{=1}$, and whose $r$-dimensional cells correspond to strata of codimension $r+1$. The attaching maps are prescribed by the inclusion relations. 
\end{definition}

For a general lc pair, the same prescription does not always yield a meaningful or well-defined cell complex; see for instance \cite[Remark 4.6]{Maurisurvey}. So the dual complex of an lc pair $(X, \Delta)$ is defined as the $\PL$-homeomorphism class of the dual complex of any dlt modification of $(X,\Delta)$, and we denote it by $\mathcal{DMR}(X, \Delta)$; the notation is an abbreviation for Dual complex of a Minimal dlt partial Resolution. 
The $\PL$-homeomorphism class $\mathcal{DMR}(X,\Delta)$ is well-defined, as it is independent of the choice of a dlt modification by~\cite[Definition 15]{deFernexKollarXu2012}.

\spa{\label{spa:minimal}
If a variety $V$ admits a \textbf{dlt compactification}, i.e.\ there exists a reduced proper dlt pair $(\overline{V}, \partial V)$ with $V = \overline{V} \setminus \partial V$, then the \textbf{dual boundary complex} of $V$ is $\mathcal{D}(\partial V)$, and it is independent of the choice of the dlt compactification up to homotopy by \cite[Theorem 3]{deFernexKollarXu2012}. If $\overline{V}$ is a dlt \textbf{logCY compactification}, i.e.\ we also require that $K_{\overline{V}}+ \partial V \sim_{\Q} 0$, then $\mathcal{D}(\partial V)$ is uniquely determined up to $\PL$-homeomorphism. Indeed, the logCY assumption fixes a distinguished class of compactifications of $V$, all crepant birational to each other: it selects a specific $\PL$-homeomorphism type for $\mathcal{D}(\partial V)$ by \cite[Proposition 11]{deFernexKollarXu2012}.
}

\spa{We refer to \cite{RourkeSanderson1982} or \cite{Lurie20092} for the elementary notions of $\PL$-topology used in this paper.}

\section{Dual complex of quotient of lc pairs}\label{sec:dualcomplexofquotient}

Let $(X, \Delta_X)$ be an lc pair. The proof of the existence of dlt modifications (cf~\cite[Corollary 1.36]{Kollar2013}) is non-constructive, which makes the explicit description of $\mathcal{DMR}(X,\Delta_X)$ a very subtle task. Inspired by~\cite[Proposition 5.6]{BoucksomJonsson2017}, we provide an alternative definition of $\mathcal{DMR}(X,\Delta_X)$ independent of the construction of a dlt modification. For some crepant quotient of dlt pairs, this gives an effective recipe to determine the $\PL$-homeomorphism type of $\mathcal{DMR}(X,\Delta_X)$.

\begin{definition}\label{definitionesssentialskeletonlogCY}
Let $(X, \Delta_X)$ be an lc pair.
The \textbf{skeleton} of $(X,\Delta_X)$ is
\begin{align*}
\Skess(X, \Delta_X) & \coloneqq \{v \in X^{\mathrm{bir}} \colon A_{(X, \Delta_X)}(v)=0 \}.
\end{align*}
The \textbf{link} of the skeleton $\Skess(X, \Delta_X)$, denoted by 
$\Skess(X, \Delta_X)^*/\R_{>0}$, is the quotient of the punctured skeleton $\Skess(X, \Delta_X)^*\coloneqq \Skess(X, \Delta_X) \setminus \{ v_0 \}$ via the rescaling action of the positive real numbers $\R_{>0}$.
\end{definition}

\begin{lemma}\label{crepantbirationalpair}
If two lc pairs $(X, \Delta_X)$ and $(Y, \Delta_Y)$ are crepant birational, then
\begin{equation}\label{equation crepant birational skeletons}
\Skess(X, \Delta_X) =\Skess(Y, \Delta_Y)  \text{ and } \Skess(X, \Delta_X)^*/\R_{>0} = \Skess(Y, \Delta_Y)^*/\R_{>0}.
\end{equation}
\end{lemma}
\begin{proof}
The equalities ~\cref{equation crepant birational skeletons} follow from the fact that $A_{(X, \Delta_X)}=A_{(Y, \Delta_Y)}$ on $X^{\text{bir}}=Y^{\text{bir}}$.
\end{proof}

\begin{lemma}\label{lem:PLstructures}
Let $(X, \Delta_X)$ be a lc pair. Given any dlt modification $(X^{\mathrm{dlt}}, \Delta^{\mathrm{dlt}})$ of $(X, \Delta_X)$, there exists a distinguished homeomorphism $i(\Delta^{\mathrm{dlt}})\colon \Cone(\mathcal{D}(\Delta^{\mathrm{dlt}})) \to \Sk(X, \Delta_X)$ so that the natural PL-structure on $\Cone(\mathcal{D}(\Delta^{\mathrm{dlt}}))$ induces a PL-structure on $\Sk(X, \Delta_X)$ which is independent of the choice of the dlt modification. In particular, there exist PL-homeomorphisms
\begin{equation}\label{eq: homeoskelDMR}
\Skess(X, \Delta_X) \simeq_{\mathrm{PL}} \Cone(\mathcal{DMR}(X, \Delta_X)) \text{ and } \Skess(X, \Delta_X)^*/\R_{>0} \simeq_{\mathrm{PL}}\mathcal{DMR}(X, \Delta_X).
\end{equation}
\end{lemma}

\begin{proof}
Let $(X^{\mathrm{dlt}}, \Delta^{\mathrm{dlt}}=\sum_{i \in I} \Delta_i)$ be a dlt modification of $(X, \Delta_X)$. Any lc centre $W$ of $(X^{\mathrm{dlt}}, \Delta^{\mathrm{dlt}})$ is an irreducible component of $\bigcap_{j \in J}\Delta_j$, for some $J \subseteq I$. At the generic point $\eta_W$ of $W$, the dlt condition guarantees that any function $f \in \mathcal{O}_{X, \eta_W}$ can be written in $\hat{\mathcal{O}}_{X, \eta_W}$ as $f=\sum_{\beta \in \N^{J}} c_{\beta} z^{\beta}$ where $z_j$ is a local equation for $\Delta_j$ at $\eta_W$. Denote by $\Cone(\sigma_W)$ the corresponding cone in $\Cone(\mathcal{D}(\Delta^{\mathrm{dlt}}))$, which we identify with the standard orthant $\R^J_{\geq 0}$ in $\R^J$. For any $\underline{w}=(w_j)_{j \in J} \in \Cone(\sigma_W)$ we associate the valuation $v_{\underline{w}} \in (X^{\mathrm{dlt}})^{\mathrm{bir}}=X^{\mathrm{bir}}$ defined by
\[v_{\underline{w}}\bigg(\sum_{\beta \in \N^{J}} c_{\beta} z^{\beta}\bigg)= \min\bigg\lbrace\sum_{j} w_j \beta_j | \, c_{\beta}\neq 0\bigg\rbrace.\]
Varying the lc centre $W$, we obtain a map 
\begin{equation}\label{eq:homeodualcomplex}
i(\Delta^{\mathrm{dlt}})\colon \Cone(\mathcal{D}(\Delta^{\mathrm{dlt}})) \to X^{\mathrm{bir}}, 
\end{equation}
which is a homeomorphism onto $\Skess(X^{\mathrm{dlt}}, \Delta^{\mathrm{dlt}})=\Skess(X, \Delta_X)$ by \cite[Proposition 5.6]{BoucksomJonsson2017}\footnote{Observe that all the valuations in \cite{BoucksomJonsson2017} are normalised by the choice of a uniformiser while we do not assume any normalisation. However, the proof of \cite[Proposition 5.6]{BoucksomJonsson2017} works in our context too, provided that one replaces $\mathcal{D}(\Delta^{\mathrm{dlt}})$ with its cone.} and \cref{crepantbirationalpair}. 

The cone complex $\Cone(\mathcal{D}(\Delta^{\mathrm{dlt}}))$ has a natural PL-structure, and so $\Skess(X, \Delta_X)$ inherits a PL-structure via the homeomorphism $i(\Delta^{\mathrm{dlt}})$. We show now that the PL-structure on $\Skess(X, \Delta_X)$ is independent of the choice of the dlt modification $(X^{\mathrm{dlt}}, \Delta^{\mathrm{dlt}})$. 
Given dlt modifications $(X^{\mathrm{dlt}}_i, \Delta^{\mathrm{dlt}}_{i})$ of $(X, \Delta_X)$, with $i=1,2$, there exists a sequence of crepant birational maps 
\[
(X^{\mathrm{dlt}}_1, \Delta^{\mathrm{dlt}}_{1}) \xleftarrow{p_1}
(Y_0, \Delta_{Y_0}) \overset{\pi_0}{\dashrightarrow} (Y_1, \Delta_{Y_1}) \overset{\pi_1}{\dashrightarrow} \ldots   (Y_{m}, \Delta_{Y_{m}}) \overset{\pi_{m}}{\dashrightarrow}  (Y_{m+1}, \Delta_{Y_{m+1}}) \ldots \overset{\pi_{r-1}}{\dashrightarrow}   (Y_r, \Delta_{Y_r}) \xrightarrow{p_2} (X^{\mathrm{dlt}}_2, \Delta^{\mathrm{dlt}}_2)
\]
such that
\begin{itemize}
    \item $(Y_i, \Delta_{Y_i})$ are snc sub-pairs, i.e.\ $\Delta_{Y_i}$ is a simple normal crossing divisor with coefficients in $(-\infty, 1]$; 
    \item the maps $\pi_i$ are smooth blow-ups or their inverse;
    \item $\mathcal{D}(\Delta^{\mathrm{dlt}}_{1})=\mathcal{D}(\Delta_{Y_0})$, $i(\Delta^{\mathrm{dlt}}_{1})=i(\Delta_{Y_0})$, $\mathcal{D}(\Delta^{\mathrm{dlt}}_{2})=\mathcal{D}(\Delta_{Y_r})$, $i(\Delta^{\mathrm{dlt}}_{2})=i(\Delta_{Y_r})$;
    \item $\mathcal{D}(\Delta_{Y_m})=\mathcal{D}(\Delta_{Y_{m+1}})$, or $\mathcal{D}(\Delta_{Y_m})$ differs from $\mathcal{D}(\Delta_{Y_{m+1}})$ by a stellar subdivision or its inverse;
\end{itemize}
see \cite[Proposition 11]{deFernexKollarXu2012} and references therein.
Note that the definition of log discrepancy, dual complex and skeleton, and \cref{crepantbirationalpair} extend to lc sub-pairs with no change. Hence, there exists a homeomorphism \[i(\Delta_{Y_m}) \colon \Cone(\mathcal{D}(\Delta_{Y_m}))\to \Sk(Y_m, \Delta_{Y_m})=\Sk(X, \Delta_{X})\]
defined as above in the context of pairs. 

If we prove that the composition $i(\Delta_{Y_{m+1}})^{-1}|_{\Sk(X, \Delta_X)} \circ i(\Delta_{Y_{m}}): \Cone(\mathcal{D}(\Delta_{Y_{m}}))\to \Cone(\mathcal{D}(\Delta_{Y_{m+1}}))$ is a PL-homeomorphism for all $0 \leq m \leq r-1$, then the PL-structures on $\Sk(X, \Delta_X)$ induced by the PL-structure of $\mathcal{D}( \Delta^{\mathrm{dlt}}_{i})$ via $i(\Delta^{\mathrm{dlt}}_i)$, with $i=1,2$, will coincide as desired. 

There are two cases: either $\mathcal{D}(\Delta_{Y_m})=\mathcal{D}(\Delta_{Y_{m+1}})$, in which case $i(\Delta_{Y_{m+1}})^{-1}|_{\Sk(X, \Delta_X)} \circ i(\Delta_{Y_{m}})$ is the identity and there is nothing to prove; or the map $\pi_m \colon (Y_{m}, \Delta^{=1}_{m}) \to (Y_{m+1},\Delta^{=1}_{m+1}=\sum_{k \in K}\Delta_{Y_{m+1},k})$, or its inverse, is a blow-up of an irreducible component $W_{m+1}$ of $\bigcap_{l \in L}\Delta_{Y_{m+1},l}$ with $L \subseteq K$, and $\mathcal{D}(\Delta_{Y_m})$ differs from $\mathcal{D}(\Delta_{Y_{m+1}})$ by a stellar subdivision corresponding to an interior point $v_E$ of the cell $\sigma_{W_{m+1}}$. In the latter case, the map $i(\Delta_{Y_{m+1}})^{-1}|_{\Sk(X, \Delta_X)} \circ i(\Delta_{Y_{m}})$ can be identified over $\Cone(\sigma_{W_{m+1}})$ with the linear map
\[\mathfrak{l}\colon \bigg\{ \prod_{l \in L}w_l=0\bigg\} \subset \R^{L \cup \{v_E\}}_{\geq 0} \to \R^{L}_{\geq 0},\]
given by $\mathfrak{l}(e_l)=e_l$ and $\mathfrak{l}(v_E)=\sum_{l \in L}e_l$, and with the identity map away from $\Cone(\sigma_{W_{m+1}})$. This says that  $i(\Delta_{Y_{m+1}})^{-1}|_{\Sk(X, \Delta_X)} \circ i(\Delta_{Y_{m}})$ is indeed a PL-homeomorphism.
\end{proof}

\begin{theorem}[Theorem F]\label{lem:quotientlogCYskeleton}
Let $(X, \Delta_X)$ be an lc pair. Let $\Gamma$ be a finite subgroup of $\Autom(X, \Delta_X)$. 
Suppose also that the quotient map $q\colon X \to X/\Gamma$ is quasi-\'{e}tale, i.e.\ the action of $\Gamma$ is free in codimension $1$. 
Then,
\begin{equation}\label{equation quotient of skeleton}
    \Skess(X/\Gamma, \Delta_{X/\Gamma}\coloneqq q_*\Delta_X)= q^{\mathrm{bir}}(\Skess(X, \Delta_{X}))\simeq_{\mathrm{\PL}} \Skess(X, \Delta_{X})/\Gamma.
\end{equation}
In particular, there exist $\PL$-homeomorphisms
\begin{align}\label{equation quotient of skeleton slice}
    \Skess(X/\Gamma, \Delta_{X/\Gamma})^*/\R_{>0} & \simeq_{\mathrm{\PL}} \Skess(X, \Delta_{X})^*/(\R_{>0}\times \Gamma),\\
    \mathcal{DMR}(X/\Gamma, \Delta_{X/\Gamma}) & \simeq_{\mathrm{PL}} \mathcal{DMR}(X, \Delta_{X})/\Gamma.\label{eq:dualcomplex}
\end{align}
\end{theorem}
\begin{proof}
Observe first that the skeleton $\Skess(X/\Gamma, \Delta_{X/\Gamma})$ is well-defined, since the pair $(X/\Gamma, \Delta_{X/\Gamma})$ is lc by ~\cite[Proposition 5.20]{KollarMori}. The group $\Gamma$ acts on a $\Gamma$-equivariant dlt modification $(X^{\mathrm{dlt}}, \Delta^{\mathrm{dlt}}) \to (X, \Delta_X)$ (cf \eqref{dltmodification}) and it permutes the lc centres of the pair $(X^{\mathrm{dlt}}, \Delta^{\mathrm{dlt}})$. Therefore, $\Gamma$ acts on $\mathcal{DMR}(X, \Delta_{X}) \simeq_{\mathrm{PL}} \mathcal{D}(\Delta^{\mathrm{dlt}})$ via PL-homeomorphism.

It is well-known that the log discrepancy function $A_{(X/\Gamma, \Delta_{X/\Gamma})}$ pulls back via $q^{\mathrm{bir}}$ to the log discrepancy function $A_{(X, \Delta_{X})}$; see \cite[Proposition 3.16]{Kollar1997}. Since $q^{\mathrm{bir}}$ is surjective, this proves the first equality of \eqref{equation quotient of skeleton}. Now, as $\Gamma \subseteq \Aut(X, \Delta_X)$, the induced $\Gamma$-action on $X^{\mathrm{bir}}$ preserves $\Skess(X, \Delta_{X})$, and $q^{\mathrm{bir}}|_{\Skess}$ is the quotient map; see \cite[Corollary 5]{Berkovich1995}.
By \cref{lem:PLmaps}, the group $\Gamma$ acts on $\Skess(X, \Delta_{X})$ via PL-homeomorphisms and the quotient $q^{\mathrm{bir}}|_{\Skess}$ is a $\PL$-map. Hence, $q^{\mathrm{bir}}|_{\Skess}$ descends to the $\PL$-homeomorphism ~\cref{equation quotient of skeleton}. 
Further, since the actions of $\Gamma$ and $\R_{>0}$ commute and the homeomorphism ~\cref{equation quotient of skeleton} is $\R_{>0}$-equivariant, we have  that ~\cref{equation quotient of skeleton slice} and \eqref{eq:dualcomplex} are $\PL$-homeomorphisms too.
\end{proof}

\begin{lemma}\label{lem:PLmaps}
Let $(Y_1, \Delta_{Y_1})$, $(Y_2, \Delta_{Y_2})$ be lc pairs, and $g: Y_1 \to Y_2$ be an algebraic morphism such that the induced map $g^{\mathrm{bir}}: Y_1^{\mathrm{bir}} \to Y_2^{\mathrm{bir}}$ restricts to a map $g^{\mathrm{bir}}|_{\Sk}: \Skess(Y_1, \Delta_{Y_1}) \to \Skess(Y_2, \Delta_{Y_2})$ on skeletons. Then $g^{\mathrm{bir}}|_{\Sk}$ is a PL-map, with respect to the PL-structures defined in \cref{lem:PLstructures}.
\end{lemma}

\begin{proof}
Let $(Y^{\mathrm{dlt}}_i, \Delta^{\mathrm{dlt}}_{Y_i})$ be a dlt modification of $(Y_i, \Delta_i)$, with $i=1,2$. For any  lc centre $W \subseteq \Delta^{\mathrm{dlt}}_{Y_i}$, denote $\Cone(\sigma_W)$ the corresponding cone in $ \Cone(\DMR(Y^{\mathrm{dlt}}_i, \Delta^{\mathrm{dlt}}_{Y_i}))\simeq \Skess(Y_i, \Delta_{Y_i})$, namely the subset of $\Skess(Y_i, \Delta_{Y_i})$ of valuations centred at $W$. For each divisor $\Delta_j$ such that $W \subseteq \Delta_j \subseteq \Delta^{\mathrm{dlt}}_{Y_i}$, choose local equation $z_{W, j, i}\in k(Y_i)^*$ at the generic point of $W$. Then the map 
\begin{equation}\label{eq:tropicalization}
    \Cone(\sigma_W) \to \R^{\codim_{Y_i}W}, \qquad v \to (\ldots, v(z_{W,j,i}), \ldots)
\end{equation}
is an embedding onto the standard orthant in $\R^{\codim_{Y_i}W}$. 
Consider now the commutative diagram 
\[
\xymatrix{
\Skess(Y_1, \Delta_{Y_1}) \ar[r]\ar[d]_{g^{\mathrm{bir}}|_{\Sk}} & \R^{m_1+m_2}\ar[d]^{\mathrm{pr}_2} & v \mapsto (\ldots, v(z_{W,j,1}),\ldots,v(z_{W,j,2}\circ g), \ldots)    \\
\Skess(Y_2, \Delta_{Y_2})\ar[r]& \R^{m_2} &  w \mapsto (\ldots, w(z_{W,j,2}),\ldots), \qquad\qquad\qquad\quad
}
\]
where the index $W$ runs over all lc centres of $\Delta^{\mathrm{dlt}}_i$, $m_i \coloneqq \sum_{W \subseteq \Delta^{\mathrm{dlt}}_{Y_i}} \codim_{Y_i} W$, and $j=1, \ldots, \codim_{Y_i} W$. 

The square is commutative by \eqref{def:Berkovich}. For any $h \in k(Y_i)^*$, the map $\Sk(Y_i, \Delta_{Y_i}) \to \R$ given by $v \mapsto v(h)$ is piecewise linear by the argument in \cite[Proposition 3.2.2]{MustataNicaise}, and so the horizontal arrows are PL-maps. In fact, they are PL-homeomorphisms onto their images since \eqref{eq:tropicalization} is an embedding. Finally, we conclude that $g^{\mathrm{bir}}|_{\Sk}$ is a PL-map, because via the square diagram it can be identified with the restriction of $\mathrm{pr}_2$.
\end{proof}

\section{The geometric P=W conjecture}\label{sec:geometricP=W}
\subsection{Preliminaries: semialgebraic neighbourhoods}
\begin{definition} \emph{\cite[Definition 3.1]{D83}} Let $M$ be a real (semi)algebraic set in $\R^n$, and $Y$ be a compact (semi)algebraic set in $M$, e.g. an algebraic subvariety of a complex projective variety $M$. A subset \(N \subset M\) is a \textbf{(semi)algebraic neighbourhood} of \(Y\) if there exist
\begin{itemize}
    \item[-]  a \textbf{rug function} \(\beta\colon M \to \R\), i.e. a proper real (semi)algebraic map such that \(\beta(x)\geq 0\) and \(Y=\beta^{-1}(0)\); 
    \item[-] a suitably small positive number \(\delta\) such that \(N = \beta^{-1}[0, \delta)\). 
\end{itemize}

The set \(N^*\coloneqq N \setminus Y\) is called a \textbf{(semi)algebraic deleted neighbourhood} of \(Y\), or simply a \textbf{neighbourhood at infinity} of $N$. The \textbf{link} of \(Y\) in \(M\) is \(\partial N \coloneqq \beta^{-1}(\delta)\).
\end{definition}

\begin{proposition}[Uniqueness of semialgebraic neighbourhoods] \emph{\cite[Proposition 3.5]{D83}}\label{uniqueness} Any two semialgebraic neighbourhoods $N_0$ and $N_1$ of \(Y\) in \(M\) are stratified isotopic, i.e.\ there exists a continuous family of homeomorphisms $h_t\colon M\to M$, with $t\in[0,1]$, preserving a Whitney stratification of $M$, and such that: (i) $h_0$ is the identity; (ii) $h_t|_Y$ is the identity for any $t$; (iii) $h_1(N_0)=N_1$.
\end{proposition}

\subsection{Dolbeault and Betti moduli spaces at infinity}
\label{subsec:defnMBMDol}
Let $C$ be a compact Riemann surface of genus $g$, and $G$ be a complex reductive group. The non-abelian Hodge correspondence is a real analytic isomorphism
\[\Psi\colon M_{\mathrm{Dol}} \to M_{\mathrm{B}}\]
between two moduli spaces parametrising $G$-local systems on $C$ in different fashions; see \cite{Simpson1994} and \cite{Simpson1997}.

The first moduli space $M_{\mathrm{Dol}}=M_{\mathrm{Dol}}(C, G)$, called \textbf{Dolbeault moduli space}, parametrises semistable $G$-Higgs bundles, i.e.\ semistable pairs $(\mathcal{E}, \phi)$ consisting of a principal $G$-bundle $\mathcal{E}$ of degree zero and a section $\phi \in H^0(C, \operatorname{ad}(\mathcal{E})\otimes K_C)$, where $K_C$ is the canonical bundle. It is endowed with a projective morphism called \textbf{Hitchin fibration}
\begin{equation}\label{eq:Hitchinfibration}
    \chi\colon M_{\mathrm{Dol}} \to \Lambda \simeq \C^{N/2} \quad (N \coloneqq \dim_{\C} M_{\mathrm{Dol}}),
\end{equation}
which is Lagrangian with respect to the canonical holomorphic
symplectic form on $M_{\mathrm{Dol}}$ given by the standard hyper-K\"{a}hler metric on $M_{\mathrm{Dol}}$; see \cite{Hitchin1987}. This implies in particular that the general fibre of $\chi$ is an abelian variety. 

The second moduli space is the \textbf{Betti moduli space}, or $G$-character variety. It is the affine GIT quotient
    \begin{equation}\label{eq:characterdefn}
        M_{\mathrm{B}}=M_{\mathrm{B}}(C, G) \coloneqq \Hom(\pi_1(C), G)\sslash G \simeq \bigg\{ (A_1, B_1, \ldots, A_{g}, B_{g}) \in G^{2g} \, \big| \, \prod^{g}_{j=1}[A_j, B_j]=1_{G} \bigg\}\sslash G,
    \end{equation}
parametrising isomorphism classes of semistable $G$-representations of the fundamental group of $C$.

The spaces $M_{\mathrm{Dol}}$ and $M_{\mathrm{B}}$ are non-proper quasi-projective varieties. The goal of the geometric P=W conjecture is to describe the asymptotic behaviour of the non-abelian Hodge correspondence $\Psi$ as we approach the boundary of algebraic compactifications of $M_{\mathrm{Dol}}$ and $M_{\mathrm{B}}$. 

The moduli space $M_{\mathrm{Dol}}$ admits a standard projective compactification denoted $\overline{M}_{\mathrm{Dol}}$. It was constructed by Hausel \cite{Hausel1998}, Simpson \cite[\S 11]{Simpson1997} and de Cataldo \cite[Theorem 3.1.1]{deCataldo21} via symplectic cut or via an equivalent algebraic quotient construction. A remarkable property of this compactification is that the Hitchin fibration $\chi$ on $M_{\mathrm{Dol}}$ has a regular extension $\overline{\chi}$ to the entire $\overline{M}_{\mathrm{Dol}}$, i.e. there exist Cartesian diagrams
\[
 \begin{tikzcd}[scale=1.1]
   \partial M_{\mathrm{Dol}}\coloneqq \overline{M}_{\mathrm{Dol}}\setminus M_{\mathrm{Dol}}  \arrow[hookrightarrow, r] \arrow[d] & \overline{M}_{\mathrm{Dol}} \arrow[d, "\overline{\chi}"] & M_{\mathrm{Dol}} \arrow[hook', l] \arrow[d, "\chi"]\\
   \partial \Lambda \arrow[hookrightarrow, r] & \overline{\Lambda}& \Lambda. \arrow[hook', l]
    \end{tikzcd}
\]
Here the boundary $\partial M_{\text{Dol}}$ and $\partial \Lambda$ are prime divisors in $\overline{M}_{\mathrm{Dol}}$ and $\overline{\Lambda}$ respectively, and $\overline{\Lambda}$ and $\partial \Lambda$ are suitable weighted projective spaces. 

Now let $\mathbb{B}^*$ be an algebraic deleted neighbourhood of $W$, and $N^*_{\mathrm{Dol}}$ be its preimage. The subset $N^*_{\mathrm{Dol}}$ is an algebraic deleted neighbourhood of $\partial M_{\mathrm{Dol}}$ defined by the pullback of the rug function defining $\mathbb{B}^*$, and so independent of the choice of $\mathbb{B}^*$ up to stratified isotopy by \cref{uniqueness}. The neighbourhood $\mathbb{B}^*$ has the homotopy type of $\mathbb{S}^{N-1}$, and indeed we may view $\mathbb{S}^{N-1}$ as the quotient of $\mathbb{B}^*$ by
radial scaling, so the Hitchin map gives a map
 \[
N^*_{\mathrm{Dol}} \xrightarrow{\chi} \mathbb{B}^* \to  \mathbb{S}^{N-1},
 \]
 that by abuse of notation we continue to denote by $\chi$.  
 
On the Betti side, there is no standard algebraic compactification of $M_{\mathrm{B}}$, and the combinatorics of the strata of a compactification of $M_{\mathrm{B}}$ is more intricate than that of the standard compactification $\overline{M}_{\mathrm{Dol}}$. However, the geometric P=W conjecture suggests that the sphere at infinity \(\mathbb{S}^{N-1}\) of the Hitchin base should be identified with the dual complex  \(\mathcal{D}(\partial M_{\mathrm{B}})\), at least when the latter is well-defined. 
To this end, we need compactifications of $M_{\mathrm{B}}$ with mild singularities. Therefore, we require the following assumption (missing in the original statement of the geometric P=W conjecture), and show in \cref{lem:dltcomp} that it is satisfied for $G=\Gl, \, \Sl$.
  
   \begin{Assumption}\label{dlt assumption}
There exists a projective reduced dlt pair \((\overline{M}_{\mathrm{B}}, \partial M_{\mathrm{B}})\) such that $ M_{\mathrm{B}}= \overline{M}_{\mathrm{B}} \setminus \partial{M}_{\mathrm{B}}$.
 \end{Assumption}
 
 Without loss of generality, we make also the following assumptions; see \cref{sec:auxiliary} and \cref{appendixindep} for a thoughtful exam of them.

 \begin{Assumption}\label{assumpdivisor}
 Any intersection of irreducible components of $\partial M_{\mathrm{B}}= \sum_{i \in I} \Delta_i$ is irreducible. 
 \end{Assumption}
 
 Let \(N^*_i \coloneqq \beta^{-1}_i(0, \delta)\) be a semialgebraic deleted neighbourhood of \(\Delta_i\), for some rug function $\beta_i$ and real number $\delta$. Then $\beta \coloneqq \min_{i \in I}\beta_i$ is a rug function for $\partial M_{\mathrm{B}}$, and so $N^*_{\mathrm{B}} \coloneqq \beta^{-1}(0, \delta) = \bigcup_{i \in I} N^*_i$ is a semialgebraic deleted neighbourhood of $\partial M_{\mathrm{B}}$. 
 
 \begin{Assumption}\label{assumpboundary} For any $J \subseteq I$, $\Delta_J \coloneqq \bigcap_{i \in J} \Delta_i$ is non-empty  
 if and only if $N_J \coloneqq \bigcap_{i \in J} N_i$ is non-empty. 
 \end{Assumption}

 Given a partition of unity \(\{\varphi_i\}_{i\in I}\) subordinate to the the open cover \(\{N^*_i\}_{i \in I}\) of $N_{\mathrm{B}}^*$, we can then define the \textbf{evaluation map} \[ev\colon N_{\mathrm{B}}^* \to \R^{I}, \quad ev(x)(i)=\varphi_i(x);\] 
see also \cite[\S 5]{EvansMauri2019}.\footnote{Given two sets $S$ and $T$, denote $S^T$ the
space of maps $T \to S$.}  
 
 \begin{Assumption}\label{independentpartition}
 The partition of unity \(\{\varphi_i\}_{i\in I}\) is \textbf{full} in the following sense.
 \begin{quote}
For any $J \subseteq I$ such that $N_J \neq \emptyset$, the map $\varphi^{J}\colon N_J \to [0,1]^{J}$, given by $\varphi^{J}(x)(j)=\varphi_j(x)$, is surjective onto the open standard simplex in $(0,1]^{J}$ given by the equation $\sum_{j \in J} y_j=1$.
 \end{quote} 
  \end{Assumption}

  Under the previous assumptions, the image of the evaluation map is the dual boundary complex \(\mathcal{D}(\partial M_{\mathrm{B}})\) as in \cite[Lemma 5.4]{EvansMauri2019} (the proof in the snc case extends to the dlt case with no change). Up to shrink $\mathbb{B}^*$, we can also suppose that $\Psi(N^*_{\mathrm{Dol}})$ embeds in $N^*_{\mathrm{B}}$, and the inclusion $\Psi(N^*_{\mathrm{Dol}})\hookrightarrow N^*_{\mathrm{B}}$ is a homotopy equivalence by \cref{lem:inclusionhomotopyequivalence}. 
\vspace{5pt}

 \begin{conjecture}[Geometric P=W conjecture]\label{conjKNPS} There exists a homotopy equivalence \(r\colon \mathbb{S}^{N-1} \to \mathcal{D}(\partial M_{\mathrm{B}})\) such that the following diagram is homotopy commutative:
 \begin{equation}\label{diagP=W}
 \begin{tikzpicture}[baseline=(current  bounding  box.east)]
 \node (A) at (0,1.5) {\(N_{\text{Dol}}^*\)};
 \node (B) at (2.5,1.5) {\(\Psi(N_{\text{Dol}}^*)\subseteq N^*_{\mathrm{B}}\)};
 \node (C) at (0,0) {\(\mathbb{S}^{N-1}\)};
 \node (D) at (3.25,0) {\(\mathcal{D}(\partial M_{\mathrm{B}}).\)};
 \draw[->,thick] (A) -- (C) node [midway,left] {\(\chi\)};
 \draw[->,thick] (3.2,1.25) -- (D) node [midway,right] {\(ev\)};
 \draw[->,thick] (A) -- (B)node [midway,above] {\(\Psi\)};
 \draw[->,thick] (C) -- (D) node [midway,above] {\(r\)};
 \end{tikzpicture}
 \end{equation}
 \end{conjecture}

We show that all the previous assumptions can be achieved in \S \ref{sec:auxiliary}. We refer to \cref{appendixindep} for the proof that the geometric P=W is independent of all the auxiliary choices made.

\subsection{Remarks on the assumptions}
\label{sec:auxiliary}
The singularities of character varieties represent a priori an obstacle to a well-posed definition of the dual boundary complex $\mathcal{D}(\partial M_{\mathrm{B}})$\, thus undermining the very formulation of the conjecture. This motivates Assumption \ref{dlt assumption}, which can be satisfied, at least when \(G\) is either \(\Gl\) or \(\Sl\). 

 \begin{lemma}[Existence of dlt compactification]\label{existence of dltcomp} Let \(X\) be an algebraic variety with \(\Q\)-factorial klt singularities, then \(X\) admits a dlt compactification.
 \end{lemma}
 \begin{proof}
 Let \(\overline{X}_1\) be any algebraic compactification of \(X\) (which exists by Nagata's compactification theorem), and \(\pi_1\colon \overline{X}_2 \to \overline{X}_1 \) be a log resolution of \(\partial X_1\coloneqq \overline{X}_1 \setminus X \). Run the \((K_{\overline{X}_2}+ \sum_i(1-\epsilon)E_i)\)-MMP following \cite{BirkarCasciniHaconEtAl}, where \(E_i\) are the \(\pi_1\)-exceptional divisors. By the \(\Q\)-factoriality, the exceptional locus is divisorial over \(X\). However, for  \(\epsilon\) sufficiently small, the MMP contracts all the exceptional divisors over \(X\), since \(X\) is klt. We conclude that the end product of the MMP, say \(\overline{X}_3 \to \overline{X}_1\), is an isomorphism over \(X\).
 Hence, \(\overline{X}_3\) is a \(\Q\)-factorial compactification of \(X\). Now we can construct a dlt modification \(\overline{X} \to \overline{X}_3\) as in \cite[Theorem 3.1]{KollarKovacs} or \cite[Theorem 4.1]{Fujino2011}. It is again an isomorphism over \(X\), thus providing the desired compactification.
 \end{proof}

 \begin{cor}\label{lem:dltcomp}
$M_{\mathrm{B}}(C, G)$ admits a dlt compactification for $G = \mathrm{GL}_n$ or $ \mathrm{SL}_n$. 
\end{cor}
\begin{proof}
By \cite[Theorems 1.2 and 1.5]{BellamySchedler2019}, $M_{\mathrm{B}}(C, G)$ has canonical and $\Q$-factorial singularities for $G = \Gl$ or $\Sl$.
\cref{existence of dltcomp} then implies the existence of dlt compactifications.
\end{proof}
 
Now let $(\overline{M}_{\mathrm{B}}, \partial M_{\mathrm{B}}= \sum_{i \in I} \Delta_i)$ be a dlt compactification of $M_{\mathrm{B}}(C,G)$, and choose neighbourhoods $N_i$ and a partition of unity as in \S \ref{subsec:defnMBMDol}. 
For any $J \subseteq I$, the evaluation map $ev$ maps the irreducible components of $\Delta_J$ to the same locus in $\R^J$, while they may correspond to different cells in the dual complex $\mathcal{D}(\partial M_{\mathrm{B}})$. To prevent this bad behaviour, we require Assumption \ref{assumpdivisor} to hold, i.e.\ that all $\Delta_J$ are irreducible, or equivalently that $\mathcal{D}(\partial M_{\mathrm{B}})$ is simplicial. This can be achieved by means of a barycentrical subdivision, and geometrically via the birational modification described in \cref{lem:dltblowup}. Applying \cref{lem:dltblowup} iteratively, we produce a new compactification of $M_{\mathrm{B}}$ which satisfies both Assumptions \ref{dlt assumption} and \ref{assumpdivisor}.

\begin{lemma}[Dlt blow-up of a stratum of a dlt pair]\label{lem:dltblowup}
Let $(X, \Delta_X)$ be a dlt pair and $W$ be a stratum. Then there exists a dlt modification $g\colon (Y, \Delta_Y) \to (X, \Delta_X)$ with a unique exceptional divisor $E_W$, and such that $g$ coincides with the blow-up of $X$ along $W$ over the snc locus $X^{\mathrm{snc}}$ of $X$. 

In particular, the dual complex $\mathcal{D}(\Delta_Y)$
is obtained from $\mathcal{D}(\Delta_X)$
by a stellar subdivision of the cell in $\mathcal{D}(\Delta_X)$
corresponding to $W$. Further, if $X$ is $\Q$-factorial, then $X \setminus \Supp(\Delta_X) \simeq Y \setminus \Supp(\Delta_Y)$.
\end{lemma}
\begin{proof}
Let $\pi_1: Y_1 \to X$ be the blow-up of $X$ along $W$ with divisorial exceptional locus $F$, and $\pi_2: Y_2 \to Y_1$ be a log resolution of $\pi_{1,*}^{-1} \Delta_X + F$ which is an isomorphism over $\pi_1^{-1}(X^{\mathrm{snc}})$. Let $E = E_W + \sum_{k \in K} E_k$ be the exceptional divisors of the map $\pi\coloneqq \pi_2 \circ \pi_1$, where $E_W$ is the only exceptional divisor dominating $W$. By construction, the generic point of $\pi(E_k)$ does not lie in $X^{\mathrm{snc}}$, and so the log discrepancy $A_{(X, \Delta_X)}(E_k)$ is positive. Note that
\[K_{Y_2} + E \sim_{\pi, \Q} \sum_{k \in K} A_{(X, \Delta_X)}(E_k) E_k,\]
and the right hand side is fully supported on $E - E_W= \sum_{k \in K} E_k$. The $\pi$-relative $(K_{Y_2} + E)$-MMP terminates with a dlt modification $g\colon (Y, \Delta_Y) \to (X, \Delta_X)$ which contracts all and only the divisors $E_k$; see \cite[\S 1.35]{Kollar2013}. Observe that $g: Y \to X$ is isomorphic to the blow-up $\pi_1: Y_1 \to X$ over $X^{\mathrm{snc}}$, as required.

Now the statement about dual complexes follows from \cite[\S 9]{deFernexKollarXu2012}. Further, the $\Q$-factoriality of $X$ implies that the exceptional locus of $g$ is divisorial, alias equal to (the strict transform of) $E_W$. Hence, the map $g$ induces an isomorphism $X \setminus \Supp(\Delta_X) \simeq Y \setminus \Supp(\Delta_Y)$.
\end{proof}

 \begin{lemma}[Assumptions \ref{assumpboundary} and \ref{independentpartition} can be achieved]\label{lem:nonemptyassumption} There exist semialgebraic neighbourhoods $N_i$ of $\Delta_i$ satisfying Assumption \ref{assumpboundary}, and a full partition of unity \(\{\varphi_i\}_{i\in I}\) subordinate to the open cover \(\{N^*_i\}_{i \in I}\).
 \end{lemma}
 
 \begin{proof}
  Given $N_i = \beta^{-1}_i[0 , \delta_i)$ for some semialgebraic rug function $\beta_i$ and $ \delta_i \in \R$, Assumption \ref{assumpboundary} can be achieved for $0<\delta_i\ll 1$.  
  Now let \(\FF_i\colon\R\to[0,1]\) be cutoff functions satisfying
  \begin{equation*} \FF_i(x)=\begin{cases} 1 & \text{ if
  }x \leq 0\\ 0 & \text{ if }x \geq \delta_i \end{cases}
  \end{equation*} and 
  which are 
  strictly decreasing on
  \([0,\delta_i)\). Choose the following partition of unity subordinate to
  \(\{N_i\}_{i\in I}\): \[ \varphi_i(x) \coloneqq  \frac{\FF_i\circ \beta_i(x)}{\sum_{i' \in
  I}\FF_{i'}\circ \beta_{i'}(x)}.\]  
  The map $\varphi^J\colon N_J \to \R^J$, given by $\varphi^J(j)=\varphi_j$, with $j \in J \subseteq I$, factors as composition of the following maps
    \begin{align}
  N_J \xrightarrow{\beta \coloneqq (\beta_j)} \prod_{j \in J}[0, \delta_j)\xrightarrow{\FF \coloneqq(\FF_j)} (0,1]^{J} & \xrightarrow{\rho}(0,1]^{J} \label{eq:evexplicit}\\
  (x_j) & \mapsto \bigg(\frac{x_j}{\sum_{j' \in J}x_{j'}}\bigg). \nonumber
      \end{align}
  Without loss of generality, we can suppose that $\beta$ is surjective. For instance, at the general point of any strata of minimal dimension we can assume that $\beta_i=|f_i|^2$, where $f_i$ is a local equation for $\Delta_i$ (eventually use semialgebraic cutoff functions to modify locally $\beta$). This is enough to ensure the surjectivity of $\beta$. Note also that $\FF$ is a diffeomorphism and $\rho$ is the radial projection onto the set $\{y\in (0,1]^{J}|\, \sum_{j \in J} y(j)=1\}.$ Hence, we observe that the functions $\varphi_i$ are full in the sense of Assumption \ref{independentpartition}.
 \end{proof}
 
 \subsection{The geometric P=W conjecture in rank one}\label{sec:P=Wrankone}
 \begin{theorem}[Theorem B]\label{thm:P=Wconjecturegeomrankonetext}
Let $C$ be a compact Riemann surface of genus $g$. Then the geometric P=W conjecture holds for $M_{\mathrm{B}}(C,\C^*)$.
\end{theorem}
\begin{proof}
Let $Y$ be a smooth projective toric variety of dimension $2g$, with toric boundary $\Delta_Y= \sum_{i \in I}\Delta_i$. There exists a symplectic moment map $\mu\colon Y \to \R^{2g}$, whose image is a polytope $\mathcal{P}$ in $\R^{2g}$, called moment polytope; see for instance \cite[\S 12.2]{CoxLittleSchenck2011}. For each toric divisor $\Delta_i \subset \Delta_Y$, let $h_i$ be a linear equation on $\R^{2g}$ which cuts out the facet $\mu(\Delta_i)$ in $\mathcal{P}$.

By definition \eqref{eq:characterdefn}, the Betti moduli space $M_{\mathrm{B}}(C,\C^*)$ is isomorphic to $(\C^*)^{2g}$, and the pair $(Y, \Delta_Y)$ is an snc logCY compactification of $ (\C^*)^{2g}$. In particular, $\mathcal{D}(\Delta_Y)$ is PL-homeomorphic to $\mathbb{S}^{2g-1}$ by \cite[Example 2.5]{Payne}. Take $r$ in \eqref{diagP=W} equal to a PL-homeomorphism, and $N_{\mathrm{B}}$ as the preimage of a semialgebraic neighbourhood of the boundary $\partial \mathcal{P}$ of the moment polytope. In order to show the homotopy commutativity of the diagram \eqref{diagP=W}, it is sufficient to prove that the maps $\chi \circ \Psi^{-1}$ and $ev$ restricted to the link $\partial N_{\mathrm{B}}$\footnote{Technically, $ev$ is not defined along $\partial N_{\mathrm{B}}$. If $N_{\mathrm{B}} = \bigcup_i N_i$, and $N_i = \beta^{-1}_i[0,\delta_i)$, then restrict $ev$ to the boundary of $\bigcup_i \beta^{-1}_i[0,\delta_i/2)$, which is isotopic to $\partial N_{\mathrm{B}}$. We omit this subtlety in the following.} 
\[\chi \circ \Psi^{-1}|_{\partial N_{\mathrm{B}}}\colon \partial N_{\mathrm{B}} \to \mathbb{S}^{2g-1} \qquad ev|_{\partial N_{\mathrm{B}}}\colon \partial N_{\mathrm{B}} \to \mathbb{S}^{2g-1}\]
are homotopic, since $N^*_{\mathrm{B}} \simeq N_{\mathrm{B}} \times \R$ retracts onto $\partial N_{\mathrm{B}}$.

Up to a diffeomorphism of the Hitchin base $\C^g$, the map $\chi \circ \Psi^{-1}\colon (\C^*)^{2g} \to \C^g\simeq \R^{2g}$ can be identified with
\[(\C^* \times \C^*)^{g} \to \R^{2g}, \qquad \Log\colon (z_1, \ldots , z_{2g}) \mapsto \bigg(-\log|z_1|, \ldots, -\log|z_{2g}|\bigg);\]
see \cite[\S 7.4]{GoldmanXia2008}\footnote{The coordinates $(\mathrm{Re}(\mathbf{p}), \mathrm{Im}(\mathbf{p}))$ in \cite[\S 7.4]{GoldmanXia2008} are real coordinates of $\C^g\simeq \R^{2g}$, while the coordinates $(\alpha_j, \beta_j)$ in loc. cit. are our coordinates $(z_{2j-1}, z_{2j})$, with $j=1, \ldots, g$.}. Again up to an orientation-preserving diffeomorphism of the Hitchin base (which is isotopic to the identity by \cite[\S 6, Lemma 2]{Milnor1997}), the $\Log$ map can be identified with the restriction of the moment map $\mu\colon  Y \to \R^{2g}$ to the big torus of $Y$ by the explicit description of $\mu$ in \cite[Eq. 12.2.4]{CoxLittleSchenck2011}.  

To construct the evaluation map $ev$, we use the compactification $(Y, \Delta_Y)$ and a partition of unity as in \cref{lem:nonemptyassumption}, where we choose $\beta_i = h_i \circ \mu$ as rug functions. In this way, $ev$ factors through $\mu$. 
Without loss of generality suppose that $\partial N_{\mathrm{B}}$ is saturated with respect to $\mu$, i.e.\ $\partial N_{\mathrm{B}}= (\min\{h_i \circ \mu\})^{-1}(\delta/2)$, and call $r'\colon  \mu(\partial N_{\mathrm{B}})\simeq \mathbb{S}^{2g-1} \to \mathcal{D}(\Delta_Y) \simeq \mathbb{S}^{2g-1}$ the continuous function such that $ev = r' \circ \mu$. The map $r'$ has topological degree one, and so by Hopf theorem it is homotopic to the identity. 

The claim about the topological degree of $r'$ follows from the construction of $ev$ in \cref{eq:evexplicit} as follows. Any torus-fixed point of $Y$ is a stratum $\Delta_J$ for $J \subset I$ with $|J|=2g$, and corresponds to an open $(2g-1)$-dimensional cell $\sigma_J$ of $\mathcal{D}(\Delta_Y)$. In the notation of \cref{eq:evexplicit}, we have
\[ev|_{ev^{-1}(\sigma_J) }= \rho \circ \mathfrak{f} \circ \mu|_{ev^{-1}(\sigma_J) } = r' \circ \mu|_{ev^{-1}(\sigma_J)}.
\]
Recall that the restriction $\rho|_{\rho^{-1}(\sigma_J)}$ is the projection \[ \max\{x_j - \mathfrak{f}_j(\delta/2) \}^{-1}(0) \subset (0,1]^J \to (0,1]^J, \qquad x_j \mapsto \frac{x_j}{\sum_{j' \in J}x_{j'}},\] which is a homeomorphism onto the simplex $\{y\in (0,1]^{J}|\, \sum_{j \in J} y(j)=1\}$. Since $\mathfrak{f}$ is a diffeomorphism, $r'$ is a homeomorphism onto $\sigma_J$, and as the open cells of dimension $2g-1$ are dense in $\mathcal{D}(\Delta_Y)$, $r'$ is generically injective. One can also check that $r'$ is smooth and preserves the orientation at a generic point of $\mathcal{D}(\Delta_Y)$, and so we conclude that $r'$ has degree one.
Note that $r'$ is not a homeomorphism globally: some fibres can be very large. For instance, $r'$ contracts the codimension zero submanifold  $\mu(\partial N_{\mathrm{B}} \cap (N_i \setminus \bigcup_{j \neq i} N_j))$ to the vertex in $\mathcal{D}(\Delta_Y)$ corresponding to $\Delta_i$. See also Figure \ref{fig:r'}. 

We conclude then that the following maps are homotopic to each other
\[\chi \circ \Psi^{-1}|_{\partial N_{\mathrm{B}}} \sim \mu|_{\partial N_{\mathrm{B}}}\sim r'\circ \mu|_{\partial N_{\mathrm{B}}}=ev|_{\partial N_{\mathrm{B}}}.\]
\end{proof}
\vspace{-0.5 cm}
\begin{figure}[H]
  \begin{tikzpicture}[scale=0.6]
  \draw[fill=yellow, opacity=0.2] (0,0)--(5,0)--(0,5)--(0,0);
  \draw[teal, fill=teal, opacity=0.30] (0,0)--(1,0)--(1,4)--(0,5)--(0,0);
  \draw[red, fill=red, opacity=0.3] (0,0)--(5,0)--(4,1)--(0,1)--(0,0);
  \draw[orange, fill=orange, opacity=0.3] (4,0)--(5,0)--(0,5)--(0,4)--(4,0);
  \draw[thick] (0.5,1)--(0.5,0.5)--(1,0.5);
  \draw[thick] (3.5,0.5)--(4,0.5)--(3.5,1);
  \draw[thick] (0.5,3.5)--(0.5,4)--(1,3.5);
  \draw[thick, dashed] (0.5,1)--(0.5, 3.5);
  \draw[thick, dashed] (1, 0.5)--(3.5, 0.5);
  \draw[thick, dashed] (3.5,1)--(1,3.5);
  \draw[thick, ->] (5.5, 2.5)--(9, 2.5);
  \draw[thick] (9.5,3.5)--(11.5,3.5)--(11.5,1.5)--(9.5,3.5);
  \node (0) at (-1.25, 4) {$\mathcal{P}$};
  \node (1) at (10.55, 4) {$\mathcal{D}(\Delta_Y)$};
  \node (A) at (2.75, -0.5) {\color{red} $h_1^{-1}[0,\delta)$};
  \node (B) at (-1.35, 2.5) {\color{teal} $h_2^{-1}[0,\delta)$};
  \node (C) at (2.75, 4) {\color{orange} $h_3^{-1}[0,\delta)$};
  \node (D) at (7.25, 3) {$r'$};
  \node (E) at (4.2,2.5) {$\mu(\partial N_{\mathrm{B}})$};
  \draw[red,fill=red] (11.5,1.5) circle (.5ex);
  \draw[teal,fill=teal] (9.5,3.5) circle (.5ex);
  \draw[orange,fill=orange] (11.5,3.5) circle (.5ex);
  \node (2) at (13.5, 2.5) {};
  \end{tikzpicture}
 \caption{\label{fig:r'}We draw the moment polytope $\mathcal{P}$ of $\PP^2$ on the left, and the dual complex $\mathcal{D}(\Delta_Y)$ of the toric boundary of $\PP^2$ on the right. The thick triangle in $\mathcal{P}$ is $\mu(\partial N_{\mathrm{B}})$. The map $r'\colon \mu(\partial N_{\mathrm{B}}) \to \mathcal{D}(\Delta_Y)$ contracts the dashed segments to the vertices of $\mathcal{D}(\Delta_Y)$, and maps a neighbourhood of the corners of $\mu(\partial N_{\mathrm{B}})$ homeomorphically onto the edges of $\mathcal{D}(\Delta_Y)$.}
\end{figure}

\begin{remark}
Observe that the non-abelian Hodge correspondence $\Psi$ in rank one exchanges fibres of the Hitchin fibration with Lagrangian tori in $(\C^*)^{2g}$ given by the equations $|z_j|=\text{constant}$, namely profound tori according to \cref{def:profoundtorus}. This should be a key ingredient for the geometric P=W conjecture in any rank, as explained in \cref{Rmk:relationotherworksII}.
\end{remark}

\section{The geometric P=W conjecture in genus one}
\subsection{Dual boundary complex of character varieties of a genus one surface} \label{sec:dualcomplexgenusone}
\spa{\label{sec:charactervarietygenusone} Let $\mathrm{E}$ be a compact Riemann surface of genus one. By \eqref{eq:characterdefn}, the Betti moduli space $M_{\mathrm{B}}(\mathrm{E},\mathrm{GL}_n)$ is isomorphic to the $n$-fold symmetric product $(\C^* \times \C^*)^{(n)}$ of the torus $\C^* \times \C^*$. Similarly, $M_{\mathrm{B}}(\mathrm{E},\mathrm{SL}_n)$ is the fibre over $\mathbf{1}=(1,1)$ of the determinant morphism 
\begin{align}\label{descriptionMSlr}
\det \colon (\C^* \times \C^*)^{(n)}\simeq M_{\mathrm{B}}(\mathrm{E},\mathrm{GL}_n) & \to \C^* \times \C^*\\
((a_i, b_i))_{i=1} ^n \simeq [(A,B)] & \mapsto (\det A, \det B)= \left(\textstyle \prod_{i=1}^n a_i, \prod_{i=1}^n b_i\right), \nonumber
\end{align}
where the pair $(A,B) \in \mathrm{GL}^2_n$ represents a point in $M_{\mathrm{B}}(\mathrm{E},\mathrm{GL}_n)$, and $(a_i)_{i=1}^n$ and $(b_i)_{i=1}^n$ are their eigenvalues; cf for instance \cite[Proposition 2.12.]{BellamySchedler2019}. The goal of this section is to compute the dual complex of a (any) dlt logCY compactification of $M_{\mathrm{B}}(\mathrm{E}, G)$ for $G = \mathrm{GL}_n, \mathrm{SL}_n$.}

\begin{theorem}[Theorem D] \label{dualcomplexhilbertscheme}
$M_{\mathrm{B}}(\mathrm{E},\mathrm{GL}_n)$ admits a dlt logCY compactification, and the dual boundary complex $\mathcal{D}(\partial M_{\mathrm{B}}(\mathrm{E},\mathrm{GL}_n))$ of such compactification is $\PL$-homeomorphic to  $\mathbb{S}^{2n-1}$.
\end{theorem}

\begin{proof}
Let $(\toric, \Delta)$ be the snc pair given by a smooth projective toric surface with its toric divisor. Note that $\toric \setminus \Delta = \C^* \times \C^*$, and so $M_{\mathrm{B}}(\mathrm{E},\mathrm{GL}_n)\simeq (Z \setminus \Delta)^{(n)}$ by \eqref{sec:charactervarietygenusone}. By \cite[Proposition 5.20]{KollarMori}, the $n$-fold symmetric product $(\toric^{(n)}, \Delta^{(n)})$ is an lc logCY compactification of $M_{\mathrm{B}}(\mathrm{E},\mathrm{GL}_n)$. 
Now, if we choose $\overline{X}_1=\toric^{(n)}$ as input compactification in the proof of \cref{existence of dltcomp}, we obtain a dlt logCY compactification of $M_{\mathrm{B}}(\mathrm{E},\mathrm{GL}_n)$ which is crepant birational to $(\toric^{(n)}, \Delta^{(n)})$.
Therefore, we have the following sequence of $\PL$-homemorphisms
\begin{alignat*}{2}
\mathcal{D}(\partial M_{\mathrm{B}}(\mathrm{E},\mathrm{GL}_n))  \simeq_{\mathrm{\PL}}  \mathcal{DMR}(\toric^{(n)}, \Delta^{(n)})
& \simeq_{\mathrm{\PL}} \mathcal{D}( \Delta^{n})/ \mathfrak{S}_n \qquad && \mathrm{ cf.\, \eqref{eq:dualcomplex}}\\
& \simeq_{\mathrm{\PL}} \mathbb{S}^{2n-1}/\mathfrak{S}_n && \text{cf.\, \cite[Example 2.5]{Payne}}\\
& \simeq_{\mathrm{\PL}} \mathbb{S}^{2n-1}. && 
\end{alignat*}

We prove that the last map is indeed a $\PL$-homemorphism. One can check that the quotient map $\C^n \to \C^{(n)}\simeq \C^n$ induces a smooth map $S_1 \coloneqq \mathbb{S}^{2n-1} \to \mathbb{S}^{2n-1} \eqqcolon S_2$, factoring through a homeomorphism $\alpha\colon  S_1/\mathfrak{S}_n \to S_2$; see \cite[Lemma 5.2.20]{Mauri2019} for the elementary details. As $\mathfrak{S}_n$ acts on $S_1$ by PL-homeomorphism, the $\PL$-structure of $\mathcal{D}(\partial M_{\mathrm{B}}(\mathrm{E},\mathrm{GL}_n))\simeq_{\PL} S_1/\mathfrak{S}_n$ is induced by a (suitably refined) triangulation of the PL-sphere $S_1$ (cf \cref{lem:quotientlogCYskeleton}), and $\alpha$ restricts on each simplex of $S_1/\mathfrak{S}_n$ to a smooth map, i.e.\ $\alpha$ can be chosen to be a piecewise differentiable homeomorphism. By a theorem of Whitehead (see \cite{Whitehead40} or \cite{Lurie20093}), every smooth manifold has a unique $\PL$-structure of $\PL$-manifold up to piecewise differentiable homeomorphism. So in our case, the $\PL$-structure of $\mathcal{D}(\partial M_{\mathrm{B}}(\mathrm{E},\mathrm{GL}_n))\simeq_{\PL} S_1/\mathfrak{S}_n$ is the standard one of the $\PL$-sphere $S_2$.
\end{proof}

\begin{remark}
Observe that the existence of a dlt compactification, namely \cref{existence of dltcomp}, is a non-constructive result. The challenge is to understand the combinatorics of the strata of the dlt boundary: this is the very reason why we adopt a non-Archimedean viewpoint in \cref{sec:dualcomplexofquotient}, instead of a purely birational one. 
\end{remark}

\begin{theorem}[Theorem E]\label{dual boundary complex SL_n} $M_{\mathrm{B}}(\mathrm{E},\mathrm{SL}_n)$ admits a dlt logCY compactification, and the dual boundary complex $\mathcal{D}(\partial M_{\mathrm{B}}(\mathrm{E},\mathrm{SL}_n))$ of such a compactification is $\PL$-homeomorphic to  $\mathbb{S}^{2n-3}$.
\end{theorem}

\begin{proof} We proceed in several steps.
\begin{enumerate}
\item[Step 1.] The Betti moduli space $M_{\mathrm{B}}(\mathrm{E},\mathrm{SL}_n)$ admits a dlt logCY compactification. 
Indeed, \eqref{descriptionMSlr} yields the diagram

\begin{equation}\label{eq:constructionMSL}
    \begin{tikzcd}
    L\coloneqq \det^{-1}(\mathbf{1}) \simeq (\C^* \times \C^*)^{n-1} \arrow[hookrightarrow]{r}{} \arrow{d}{} & (\C^* \times \C^*)^{n} \arrow[swap]{d}{q} \arrow{r}{\det} & \C^* \times \C^* \ni \mathbf{1}=(1,1) \\
    M_{\mathrm{B}}(\mathrm{E},\mathrm{SL}_n) \arrow[hookrightarrow]{r} & (\C^* \times \C^*)^{(n)}, \arrow{ru}{} & 
    \end{tikzcd}
\end{equation}
where 
$q$ the quotient by the action of the symmetric group $\mathfrak{S}_n$, permuting the factors. The projective closure $\overline{L}$ of $L$ in $(\PP^1 \times \PP^1)^{n}$ is invariant with respect to the action 
\[(\C^* \times \C^*)^{n-1} \times (\PP^1 \times \PP^1)^{n} \to (\PP^1 \times \PP^1)^n\] 
given by 
\begin{align*}
((a_i, & b_i))_{i=1}^{n-1}  \cdot  ([x_j: y_j], [z_j:w_j])_{j=1}^{n}= \Big( [a_1x_1: y_1],[b_1z_1:w_1], \ldots\\
&  \ldots , [a_{n-1}x_{n-1}: y_{n-1}],[b_{n-1}z_{n-1}:w_{n-1}], 
\,\big[
\textstyle 
\prod_{i=1}^{n-1} a_i^{-1} \cdot x_n: y_n],[\prod_{i=1}^{n-1} b_i^{-1} \cdot z_n : w_n\big] \Big).
\end{align*} 

\noindent As $L\simeq (\C^* \times \C^*)^{n-1}$ is a dense orbit of this algebraic action, it follows that $\overline{L}$ is a toric compactification of $L$. 
In particular, the pair $(\overline{L}, \partial L \coloneqq \overline{L} \setminus L)$ is a (normal) lc logCY pair. Although $\overline{L}$ is a toric variety, it is worth pointing out that the embedding $\overline{L} \hookrightarrow (\mathbb{P}^1 \times \mathbb{P}^1)^n$ is not toric.

Since $\overline{L}$ is $\mathfrak{S}_{n}$-invariant and the restriction of the quotient map $q\colon  (\mathbb{P}^1 \times \mathbb{P}^1)^n \to (\mathbb{P}^1 \times \mathbb{P}^1)^{(n)}$ to $\overline{L}$ is quasi-\'{e}tale, the projective closure $\overline{M}$ of $M_{\mathrm{B}}(\mathrm{E},\mathrm{SL}_n)$ in $(\mathbb{P}^1 \times \mathbb{P}^1)^{(n)}$ is an lc logCY compactification of $M_{\mathrm{B}}(\mathrm{E},\mathrm{SL}_n)$. As in \cref{dualcomplexhilbertscheme}, via \cref{existence of dltcomp}, we obtain a dlt compactification of $M_{\mathrm{B}}(\mathrm{E},\mathrm{SL}_n)$ crepant birational to $(\overline{M}, \Delta_{\overline{M}}\coloneqq \overline{M} \setminus M_{\mathrm{B}}(\mathrm{E},\mathrm{SL}_n))$, and whose dual boundary complex is by \cref{crepantbirationalpair}
$$\mathcal{D}(\partial M_{\mathrm{B}}(\mathrm{E},\mathrm{SL}_n)) 
\simeq_{\PL} \mathcal{DMR}(\overline{M}, \Delta_{\overline{M}})
\simeq_{\PL} \Skess(\overline{M}, \Delta_{\overline{M}})^*/\R_{>0}.
$$

\item[Step 2.] Let $\Delta$ be the toric boundary of $\mathbb{P}^1 \times \mathbb{P}^1$ and $N$ be the cocharacter lattice of the torus $\C^* \times \C^* \subseteq \mathbb{P}^1 \times \mathbb{P}^1$. 
It follows from \cref{eq: homeoskelDMR} that $\Skess(\mathbb{P}^1 \times \mathbb{P}^1, \Delta) \simeq \R^2 \simeq N_\R$.
The determinant morphism is a toric morphism, and it induces a map $\alpha_n \colon \Skess(\mathbb{P}^1 \times \mathbb{P}^1, \Delta)^n \to \Skess(\mathbb{P}^1 \times \mathbb{P}^1, \Delta)$ which can be identified with the linear map
\begin{align*}
    (N_\R)^n \simeq \R^{2n} \to N_\R \simeq \R^2, \qquad 
    (x_i, y_i)_{i=1}^n \mapsto \bigg( \textstyle \sum_{i=1}^n x_i, \sum_{i=1}^n y_i\bigg).
\end{align*}
Finally, observe that the symmetric quotient of the kernel of $\alpha_n$ is isomorphic to the additive group $\C^{n-1}$, i.e.\
$\alpha_n^{-1}(\mathbf{0})/\mathfrak{S}_{n} \simeq \C^{n-1}$.
This follows from the diagram below:
\begin{center}
    \begin{tikzcd}
    \alpha_n^{-1}(\mathbf{0}) \arrow[hookrightarrow]{r}  \arrow{d} & \Skess((\mathbb{P}^1 \times \mathbb{P}^1)^{n}, \Delta^n) \simeq \C^n \simeq \R^{2n} \arrow{r}{\alpha_n} \arrow[swap]{d}{q^{\mathrm{bir}}} & \Skess(\mathbb{P}^1 \times \mathbb{P}^1, \Delta) \simeq \C \simeq \R^2 \\
    \alpha_n^{-1}(\mathbf{0})/\mathfrak{S}_{n} \arrow[hookrightarrow]{r}  & \Skess((\mathbb{P}^1 \times \mathbb{P}^1)^{(n)}, \Delta^{(n)})\simeq \C^n, \arrow[swap]{ru}{\mathrm{pr}} & 
    \end{tikzcd}
\end{center}
where the map $\mathrm{pr}$ is the linear projection to the $\mathfrak{S}_n$-invariant coordinate $\alpha_n$.
\item[Step 3.] 
The skeleton of the pair $(\overline{M}, \Delta_{\overline{M}})$ is $\PL$-homeomorphic to the symmetric quotient of the fibre of $\alpha_n$: this can be shown following the same strategy of \cite[Proposition 6.3.3]{BrownMazzon}; see also the final paragraph of \cref{dualcomplexhilbertscheme}. We conclude that
$$\mathcal{D}(\partial M_{\mathrm{B}}(\mathrm{E},\mathrm{SL}_n)) 
\simeq_{\PL} \Skess(\overline{M}, \Delta_{\overline{M}})^*/\R_{>0}
\simeq_{\PL} (\alpha^{-1}_{n}(\mathbf{0}))^* /(\mathfrak{S}_n \times \R_{>0})
\simeq_{\PL} \mathbb{S}^{2n-3}.
$$
\end{enumerate}
\end{proof}

\subsection{A degeneration of Hilbert schemes on K3 to $M_{\mathrm{B}}(\mathrm{E},\mathrm{GL}_n)$} \label{alternativeproofindiscretelyvalued}
\spa{
The proofs of~\cref{sec:dualcomplexgenusone} are inspired by the result in~\cite[Proposition 6.2.4]{BrownMazzon}. There, Brown and the second author show that the dual complex of a degeneration of the Hilbert scheme of $n$ points of K3 surfaces induced by a maximally degenerate degeneration of K3 surfaces is PL-homeomorphic to the projective space $\P^n(\C)$. 

In this section, we exhibit a direct connection between the two results: 
we show how~\cref{dualcomplexhilbertscheme} and \ref{dual boundary complex SL_n} can be deduced from~\cite[Proposition 6.2.4]{BrownMazzon}.
This alternative proof relies on 
the construction of explicit degenerations of hyper-K\"{a}hler varieties (see~\cref{degenerationHIlbertscheme} and \ref{construct dlt compact}), and 
a global-to-local argument (\cref{linkdualcomplex}) that relates the dual complex of a degeneration to that of an lc logCY pair. 
While the proofs of~\cref{dualcomplexhilbertscheme} and \ref{dual boundary complex SL_n} presented in ~\cref{sec:dualcomplexgenusone} are technically more elementary,
we expect both strategies to prove useful for future calculations of dual complexes, as explained in \cref{introsec:hyperkahler}.
}

\begin{definition}\emph{\cite[\S 4.18]{Kollar2013}}
Let $(X, \Delta_X)$ be a dlt pair with  $\Delta_X^{=1}=\sum_{i}\Delta_i$. For any stratum $W$ of $(X,\Delta_X)$, there exists a unique $\Q$-divisor $\Diff^*_W(\Delta_X)$, called \textbf{different}, with the following property: for $m \in \N$ divisible enough, the Poincar\'{e} residue map $
    \omega^{[m]}_X(m \Delta_X)|_W \simeq \omega^{[m]}_W$ on the snc locus extends to the isomorphism $\omega^{[m]}_X(m \Delta_X)|_W \simeq \omega^{[m]}_W(\Diff^*_W(\Delta_X))$ on the locus where $\omega^{[m]}_X(m \Delta_X)|_W$ and $\omega^{[m]}_W$ are locally free and hence everywhere, being reflexive sheaves. 
\end{definition}
In particular, this yields
    $$(K_X + \Delta_X)|_W \sim_{\Q} K_W + \Diff^*_W(\Delta_X),$$ and by adjunction we have that 
  \begin{equation} \label{dualcomplexstrataI} \Diff^*_W(\Delta_X)= \Diff^*_W(\Delta_X)^{=1}+ \Diff^*_W(\Delta_X)^{<1} = \sum_{i\colon W \nsubseteq \Delta_i }\Delta_i|_W + \Diff^*_W(\Delta_X)^{<1}.
  \end{equation}
  In particular, any stratum $W$ of a dlt (logCY) pair has an induced structure of dlt (logCY) pair $(W, \Diff^*_W(\Delta_X))$ such that  \begin{equation}\label{dualcomplexstrata}
  \mathcal{D}(\Diff^*_W(\Delta_X)^{=1})\simeq \mathcal{D}\bigg(\sum_{i \colon W \nsubseteq \Delta_i }\Delta_i|_W\bigg).
  \end{equation}

\begin{lemma}[Global-to-local argument]\label{linkdualcomplex}
  Let $(X, \Delta_X)$ be a dlt pair such that the dual complex of $\mathcal{D}(\Delta_X)$ is a $\PL$-manifold. For any stratum $W$ of $\Delta_X^{=1}$,  $\mathcal{D}(\Diff^*_{W} (\Delta_X))$ is $\PL$-homeomorphic to a sphere.
  \end{lemma}
  \begin{proof}
  If $\sigma_W$ is a cell of $\mathcal{D}(\Delta_X)$ corresponding to the stratum $W$ we define: 
  \begin{itemize}
      \item the (open) star of $\sigma_W$, denoted $\St(\sigma_{W})$, as the union of the interiors of the
cells whose closure intersects the cell $\sigma_W$;
\item the closed star of $\sigma_W$, denoted $\St(\sigma_{W})$, is the closure of $\St(\sigma_{W})$;
\item the link of $\sigma_W$, denoted $\Link(\sigma_W)$, is the difference $\Stcl(\sigma_{W}) \setminus St(\sigma_{W})$ in the first
barycentrical subdivision of $\mathcal{D}(\Delta_X)$.
  \end{itemize}
By \eqref{dualcomplexstrata}, the vertices of $\Link(\sigma_W)$ are in correspondence with the strata of $\Diff^*_W(\Delta_X)^{=1}$, and $\Link(\sigma_W)$ is PL-homeomorphic to $\mathcal{D}( \Diff^*_W(\Delta_X))$. Since $\mathcal{D}(\Delta_X)$ is a $\PL$-manifold,  $\Link(\sigma_W)$ is $\PL$-homeomorphic to a sphere.
  \end{proof}
  
\spa{
We will construct a degeneration of Hilbert schemes on a K3 surface such that a component of the special fibre, paired with the different of the special fibre, is crepant birational to a dlt compactification of $M_{\mathrm{B}}(\mathrm{E},\mathrm{GL}_n)$.
Then we combine this construction with the global-to-local argument to compute $\mathcal{D}(\partial M_{\mathrm{B}}(\mathrm{E},\mathrm{GL}_n))$.
The properties of the required degeneration are collected below.
}
  
\begin{definition}\label{good minimal model}
Let $X$ be a normal variety over $\C((t))$. A \textbf{model} of $X$ over $\C[[t]]$ is a normal, flat, separated scheme $\cX$ of finite type over $\C[[t]]$ endowed with an isomorphism of $\C((t))$-schemes $\cX_{\C((t))} \simeq X$. 
\\A model $\cX$ is \textbf{defined over a curve} if there exist the germ of a smooth curve $C$, a closed point $0 \in C$, an isomorphism $\widehat{\O}_{C,0} \simeq \C[[t]]$, a normal variety $Y$, and a flat morphism $Y \to C$ such that
$\cX \simeq Y \times_C \Spec( \widehat{\O}_{C,0})$.
\\If $K_X$ is semiample, then a model $\cX$ is \textbf{good minimal dlt} if $\cX$ is $\Q$-factorial, the pair $(\cX, \cX_{0,\text{red}})$ is dlt, and $K_{\cX}+ \cX_{0,\text{red}}$ is semiample. Further, $X$ is \textbf{maximally degenerate} if the dual complex of the special fibre of a good minimal dlt model has dimension $\dim_{\C((t))}X$.
\end{definition}
  
\begin{theorem}\label{dualcomplexirreduciblecompdegeneration}
Let $(X,\Delta_X)$ be a projective lc logCY pair. Let $S$ be a maximally degenerate $K3$ surface over $\C((t))$ admitting a good minimal dlt model with reduced special fibre.
Assume there exists a good minimal dlt model $\cY$ of the Hilbert scheme of $n$ points of $S$,
such that $(X,\Delta_X)$ is crepant birational to $(D,\mathrm{Diff}^*_D(\cY_{0,\mathrm{red}}))$ for some irreducible component $D$ of the special fibre $\cY_0$.
Then, $\mathcal{D}(\Delta_X)$ is $\PL$-homeomorphic to a sphere.
\end{theorem}

\begin{proof}
By \cite[Proposition 11]{deFernexKollarXu2012}, the dual complexes of dlt crepant birational pairs are PL-homeomorphic, and so $\mathcal{D}(\Delta_X)\simeq_{\mathrm{PL}}\mathcal{D}(\mathrm{Diff}^*_D(\cY_{0,\mathrm{red}}))$. Therefore, by Lemma \ref{linkdualcomplex}, it is sufficient to show that $\mathcal{D}(\cY_0)$ is a PL-manifold. To this end, observe that $\mathcal{D}(\cY_0)$ is homeomorphic to a complex projective space by \cite[Proposition 6.2.4]{BrownMazzon}. The same argument of the proof of \cref{dualcomplexhilbertscheme} gives that the homeomorphism in \cite[Proposition 6.2.4]{BrownMazzon} is piecewise differentiable, and so $\mathcal{D}(\cY_0)$ is actually a PL-manifold; see \cite{Whitehead40} or \cite{Lurie20093}.
\end{proof} 

\begin{prop}[Theorem G]\label{degenerationHIlbertscheme}
There exists a good minimal dlt model $\mathscr{S}^{[n],\dlt}$ of the Hilbert scheme of $n$ points of a K3 surface, and an irreducible component $\Delta^{\dlt}_i$ of the special fibre $\mathscr{S}^{[n],\dlt}_{0}$ such that the pair $(\Delta_i^{\mathrm{dlt}} ,\Diff^*_{\Delta^{\mathrm{dlt}}_i}(\cS^{[n], \mathrm{dlt}}_{0, \mathrm{red}}))$
is crepant birational to an lc logCY compactification of $M_{\mathrm{B}}(\mathrm{E},\mathrm{GL}_n)$.
\end{prop}
    
\begin{proof}
Let $\cS$ be the following model of a quartic $K3$ surface $S$
\[
   \cS\coloneqq \left\{x_0x_1x_2x_3 +t \sum^{3}_{i=0}x^4_i=0\right\} \subseteq \PP^3_{[x_0:x_1:x_2:x_3]}\times \Spec(\C[[t]]);
 \]
it is a dlt logCY model defined over a curve, with reduced special fibre.
The relative $n$-fold symmetric product $(\cS^{(n)}, \cS^{(n)}_0)$ is a reduced lc logCY pair, because it is a quasi-\'{e}tale quotient of the reduced lc logCY pair $(\cS^{n}, \cS^{n}_0)$.
Let $(\cS^{[n]}, \cS^{[n]}_0)$ be the relative Hilbert scheme of $n$ points on $\cS$, and 
\[
\rho_{HC}\colon (\cS^{[n]}, \cS^{[n]}_0) \longrightarrow (\cS^{(n)}, \cS^{(n)}_0)
\]
be the relative Hilbert--Chow morphism, which is a crepant birational map on the generic fibre; see \cite[\S 6]{Beauville1983}.  

Consider now a log resolution of $(\cS^{[n]}, \cS^{[n]}_0)$, written
  \[g\colon (\cX, \Delta_{\cX}\coloneqq
   g^{-1}_*\cS^{[n]}_{0}+E) \longrightarrow (\cS^{[n]}, \cS^{[n]}_0),\]
where $E$ is the sum of the $g$-exceptional divisors, and such that $g$ is an isomorphism on the snc locus of $(\cS^{[n]}, \cS^{[n]}_0)$.
The $(K_{\cX/\C[[t]]}+ \Delta_{\cX})$-MMP with scaling terminates with a $\Q$-factorial dlt minimal model of $S^{[n]}$  
     \[h\colon  (\cS^{[n], \text{dlt}},  
     \cS^{[n],\dlt}_{0, \mathrm{red}} = h^{-1}_*\cS^{(n)}_0+E')
     \longrightarrow (\cS^{(n)}, \cS^{(n)}_0),\] 
where $E'$ is the sum of the $( \rho_{\mathrm{HC}} \circ g)$-exceptional divisors that lie in the special fibre, and $\cS^{[n],\dlt}_{0, \mathrm{red}}$ is the reduced special fibre of $\cS^{[n],\dlt}$. The existence of such $h$ follows from~\cite[Corollary 1.36]{Kollar2013}; note that the degeneration $\cS$ is defined over a curve, so we can run a relative MMP as usual.
The pair $(\cS^{[n], \text{dlt}}, \cS^{[n], \text{dlt}}_{0, \mathrm{red}})$ is logCY as well, as $h$ is a crepant morphism of pairs; see \cite[\S 1.35]{Kollar2013}. Hence, $\cS^{[n], \text{dlt}}$ is a good minimal dlt model of $S^{[n]}$.
  
Note that the special fibre $\cS^{(n)}_0$ contains irreducible components $\Delta_i \simeq (\PP^2)^{(n)}$, for $i\in \{0, \ldots, 3\}$, which are the $n$-fold symmetric products of the hyperplanes $D_i\coloneqq \{x_i=t=0\}$. 
Choose $\Delta^{\dlt}_i$ the strict transform of $\Delta_i$ under $h$.
  By~\cite[Proposition 4.6]{Kollar2013}, the following lc logCY pairs are crepant birational:
  \[
  (\Delta^{\dlt}_i, \Diff^*_{\Delta^{\dlt}_i}(\cS^{[n], \dlt}_{0, \mathrm{red}})) \sim (\Delta_i, \Diff^*_{\Delta_i}(\cS^{(n)}_0)).
  \]
  We have that $M_{\mathrm{B}}(\mathrm{E},\mathrm{GL}_n)$ is isomorphic to 
  \[
  \Delta^\circ_i\coloneqq (D_{i}\setminus \cup_{j\neq i} D_{j})^{(n)} \simeq (\C^* \times \C^*)^{(n)} \subset \Delta_i.
  \]
  We conclude that $(\Delta_i, \Diff^*_{\Delta_i}(\cS^{(n)}_0))$ is a compactification of $M_{\mathrm{B}}(\mathrm{E},\mathrm{GL}_n)$ by \cref{lemma for different}. 
  \end{proof}

\begin{lemma}\label{lemma for different}
$\Diff^*_{\Delta_i}(\cS^{(n)}_{0}) = \Diff^*_{\Delta_i}(\cS^{(n)}_{0})^{=1} = \Delta_i \setminus \Delta^\circ_i$.
\end{lemma} 
\begin{proof}
By construction, the divisor $\cS^{(n)}_0-\Delta_i$ intersects $\Delta_i$ along the complement $\Delta_i \setminus \Delta^\circ_i$; thus by \cref{dualcomplexstrataI} $\Diff^*_{\Delta_i}(\cS^{(n)}_{0}) \geqslant \Delta_i \setminus \Delta^\circ_i$. For the equality, 
it is sufficient to prove that $\cS^{(n)}$ and $\Delta_i$ are regular along the generic point $D$ of any divisor in $\Delta^\circ_i$; see \cite[Proposition 4.5 (1)]{Kollar2013}. Now, $D$ is not contained in the singular locus of $\Delta_i$ by dimensional reason. Therefore, locally at $D$, the central fibre $\cS^{(n)}_0 \stackrel{\text{loc. at }D}{=} 
\Delta_i$ is smooth and Cartier in $\cS^{(n)}$, so $\cS^{(n)}$ and $\Delta_i$ are regular along $D$, as required.
\end{proof}

\begin{proof}[Alternative proof of Theorem \ref{dualcomplexhilbertscheme}] In view of \cref{degenerationHIlbertscheme} and \cref{dualcomplexirreduciblecompdegeneration}, we have that
$$
\mathcal{D}(\partial M_{\mathrm{B}}(\mathrm{E},\mathrm{GL}_n))  
\simeq_{\mathrm{\PL}} \mathcal{DMR}(\Delta_i^{\mathrm{dlt}} ,\Diff^*_{\Delta^{\mathrm{dlt}}_i}(\cS^{[n], \mathrm{dlt}}_{0, \mathrm{red}})) 
\simeq_{\mathrm{\PL}} \mathbb{S}^{2n-1}.$$
\end{proof}

\subsection{A degeneration of generalised Kummer surfaces to $M_{\mathrm{B}}(\mathrm{E},\mathrm{SL}_n)$}\label{alternativeproofdiscretevalued2}
\spa{Following the same strategy as in \cref{alternativeproofindiscretelyvalued}, one can invoke the global-to-local argument to reduce the proof of Theorem \ref{dual boundary complex SL_n} to the construction of a suitable degeneration. Observe that the role of the Hilbert scheme on a K3 surface in~\cref{dualcomplexirreduciblecompdegeneration} is replaced by the generalised Kummer variety of an abelian surface.}

\begin{prop}[Theorem G]\label{construct dlt compact}
There exist a good minimal dlt model $\cK^{\dlt}_{n-1}$ of a generalised Kummer variety and an irreducible component $\Delta^{\dlt}$ of the special fibre $\cK^{\dlt}_{n-1, 0}$ such that the pair $(\Delta^{\dlt}, \Diff^*_{\Delta^{\dlt}}(\cK^{\dlt}_{n-1, 0, \operatorname{red}}))$ is crepant birational to an lc logCY compactification of $M_{\mathrm{B}}(\mathrm{E},\mathrm{SL}_n)$.
\end{prop}

\begin{proof}
Let $\mathscr{E}$ be the following degeneration of an elliptic curve $E$ with multiplicative reduction    
\[
\mathscr{E} \coloneqq \{ xyz + t(x^3+y^3+z^3) = 0 \} \subset \PP^2_{[x:y:z]}\times \Spec(\C[[t]]) \,;
 \]
it is a good minimal snc model defined over a curve, with reduced special fibre. The N\'{e}ron model $\cN$ of $\cE$ is the group scheme obtained from $\cE$ by removing the nodes of the special fibre; see \cite[Theorem 10.2.14]{Liu2002}. The multiplication on $\cN$ induces the morphism
\begin{align*}
    m_n\colon (\cN \times \cN)^n \coloneqq  (\cN \times_{\C[[t]]} \cN) \times_{\C[[t]]} \ldots \times_{\C[[t]]} (\cN \times_{\C[[t]]} \cN) & \longrightarrow (\cN \times_{\C[[t]]} \cN),\\
    ((n_1,n_2), \ldots, (n_{2n-1}, n_{2n})) & \mapsto \bigg(\prod^{n}_{i=1}n_{2i-1},\prod^{n}_{i=1}n_{2i}\bigg).
\end{align*}
Let $\cX_{n-1}$ be the inverse image under $m_n$ of the identity section of $\cN \times_{\C[[t]]} \cN$, and let $\overline{\cX}_{n-1}$ be the closure  of $\cX_{n-1}$ in $(\cE \times \cE)^n$; this is invariant under the action of the symmetric group $\mathfrak{S}_n$, which acts by permuting the factors of $(\cE \times \cE)^n$. 
The quotient \[\cK^{\mathrm{sing}}_{n-1} \coloneqq \overline{\cX}_{n-1}/\mathfrak{S}_n.\]
is a model of the singular generalised Kummer variety $K^{\mathrm{sing}}_{n-1}(E \times E)$ associated to the abelian surface $E \times E$:
\begin{equation} \label{equ: diagram Ksing}
\arraycolsep=2pt 
\begin{array}{ccccccc}
     & & \cX_{n-1} & \subset & (\cN \times \cN)^n & &  \\
     & & \cap &  & \cap &  &  \\
     & & \overline{\cX}_{n-1} & \subset &  (\mathscr{E} \times \mathscr{E})^n & & \\
     & & \downarrow &  &  &  & \\
     K^{\mathrm{sing}}_{n-1}(E \times E) & \subset & \cK^{\mathrm{sing}}_{n-1}. &  &  &  &
\end{array}
\end{equation}
We denote by $\cK^{\mathrm{sing}}_{n-1,0}$ the special fibre of $\cK^{\mathrm{sing}}_{n-1}$.

\begin{lemma} \label{lem:reduced lc logCY}
The pair $(\cK^{\mathrm{sing}}_{n-1},\cK^{\mathrm{sing}}_{n-1,0})$ is reduced lc logCY.
\end{lemma}
The proof of \cref{lem:reduced lc logCY} relies on local computations on the Tate curve, the uniformisation of $\mathscr{E}$; see \cref{appendixA} for details.
\begin{proof}
We omit the subscript $n-1$ for brevity. Since the quotient map $ \overline{\cX}\to \cK^{\mathrm{sing}}$ is quasi-\'{e}tale, it is equivalent to prove that the pair $(\overline{\cX}, \overline{\cX}_{0})$ is reduced lc logCY. By \cref{cor:appendix}, $(\overline{\cX}, \overline{\cX}_{0})$ is reduced lc.
Further, $\overline{\cX} \setminus \cX$ has codimension two in $\overline{\cX}$, again by \cref{cor:appendix}. Therefore, to verify that $K_{\overline{\cX}/\C[[t]]}+\overline{\cX}_0$ is trivial, it suffices to check that the restriction to $\cX$
$$
\left(K_{\overline{\cX}/\C[[t]]}+\overline{\cX}_0 \right)|_{\cX} = K_{\cX/\C[[t]]} + \cX_0 \sim K_{\cX/\C[[t]]}
$$ 
is trivial. 
To this end, denote by $N_{\cX}$ the normal bundle of $\cX$ in $(\cN \times \cN)^n$. As $\cX$ is a fibre of the locally trivial fibration $m_n$, $\det (N_{\cX})$ is trivial. In particular, we have
    \[ 
    K_{\cX/\C[[t]]} \sim K_{(\cN \times \cN)^n /\C[[t]]}|_{\cX} \otimes \det (N_{\cX}) \sim 0,
    \]
since $(\cN \times \cN)^n$ is Calabi--Yau. Thus, the pair $(\overline{\cX}, \overline{\cX}_{0})$ is logCY, as required.
\end{proof}
   
We have shown that, given $\cN$ as input, the above construction yields the lc logCY pair $(\cK^{\mathrm{sing}}_{n-1},\cK^{\mathrm{sing}}_{n-1,0})$. Now restricting the previous construction to the identity component of the special fibre of $\cN$, which is isomorphic to $\C^*$, we obtain \eqref{eq:constructionMSL}. As a consequence, an irreducible component of the special fibre $\cK^{\mathrm{sing}}_{n-1, 0}$ is the lc logCY compatification $\overline{M}$ in $(\PP^1 \times \PP^1)^{(n)}$ constructed in Step 1 of the proof of \cref{dual boundary complex SL_n}.
   
Finally, a good minimal dlt model $\cK^{\dlt}_{n-1}$ of the generalised Kummer variety $K_{n-1}(E \times E)$ can be obtained from $\cK^{\mathrm{sing}}_{n-1}$, by extracting the exceptional divisors of the Hilbert--Chow morphism
$\rho_{\mathrm{HC}}\colon K_{n-1}(E \times E) \to  K^{\mathrm{sing}}_{n-1}(E \times E)$; see \cite[Corollary 1.38]{Kollar2013}.
Exactly as in the proof of \cref{degenerationHIlbertscheme}, there exists then an irreducible component $\Delta^{\dlt}$ of the special fibre $\cK^{\dlt}_{n-1,0}$ such that the pair $(\Delta^{\dlt}, \Diff^*_{\Delta^{\dlt}}(\cK^{\dlt}_{ 0, n-1,\operatorname{red}}))$ is crepant birational to an lc logCY compactification of $M_{\mathrm{B}}(\mathrm{E},\mathrm{SL}_n)$.
\end{proof}

\begin{proof}[Alternative proof of~\cref{dual boundary complex SL_n}]
It follows from ~\cref{construct dlt compact}, \cite[Proposition 6.3.4]{BrownMazzon} and \cref{linkdualcomplex}.
\end{proof}

\subsection{The full geometric P=W conjecture in genus one} \label{sec:fullP=Wgenusone}
\begin{theorem}[Theorem B]\label{thm:P=Wconjecturegeomgenusonetext}
Let $\mathrm{E}$ be a compact Riemann surface of genus one, and $G$ be either $\mathrm{GL}_n$ or $\mathrm{SL}_n$. Then the geometric P=W conjecture holds for $M_{\mathrm{B}}(\mathrm{E},G)$.
\end{theorem}

\begin{proof}
We prove the statement for $G=\Gl$; the proof works for $G=\Sl$ mutatis mutandis. The geometric P=W conjecture in genus one consists of a subtle generalisation of the proof of \cref{thm:P=Wconjecturegeomrankonetext}.

By \cref{dualcomplexhilbertscheme} we can identify the dual boundary complex $\mathcal{D}(\partial M_{\mathrm{B}}(E,\mathrm{GL}_n))$  with $\mathbb{S}^{2n-1}$. Then, as in \cref{thm:P=Wconjecturegeomrankonetext}, it is sufficient to prove that the maps $\chi \circ \Psi^{-1}$ and $ev$ restricted to the link $\partial N_{\mathrm{B}}$ are homotopic 
\[\chi \circ \Psi^{-1}|_{\partial N_{\mathrm{B}}}\colon  \partial N_{\mathrm{B}} \to \mathbb{S}^{2n-1}, \qquad ev|_{\partial N_{\mathrm{B}}}\colon  \partial N_{\mathrm{B}} \to \mathbb{S}^{2n-1}.\]
\begin{enumerate}
\item[Step 1.] The goal is to choose the map $ev$ matching locally the moment map of a toric variety. To this end, we construct a special dlt logCY compactification $(\overline{X}_3, \Delta_3)$ of $M_{\mathrm{B}}(\mathrm{E},\Gl)$.

Let $(\toric, \Delta)$ be a smooth projective toric surface pair with at least $n$ torus-fixed point. Let $z \in \toric^{n}$ be an $n$-tuple of $n$ distinct fixed points, and $q\colon  \toric^{n}\to \toric^{(n)}$ be the quotient map to the symmetric product $\toric^{(n)}$. The pair $(\toric^{(n)}, \Delta^{(n)})$ has normal crossings at $q(z)$, since $(\toric^{n}, \Delta^{n})$ is an snc pair and the map $q$ is \'{e}tale over $q(z)$. 
However, the strata of $\Delta^{(n)}$ meeting at $q(z)$ have self-intersections. To avoid this, a suitable sequence of blowups along (the strict transform of) lc centers of $(\toric^{(n)}, \Delta^{(n)})$ through $q(z)$ yields a birational morphism 
\[\pi\colon  (\overline{X}_1 \supset \Delta_1\coloneqq \pi^{-1}(\Delta^{(n)})^{\mathrm{red}}) \to (\toric^{(n)}, \Delta^{(n)})\] 
and a 0-dimensional stratum $x_1$ in $\Delta_1$, with the following properties: (i) $\pi(x_1)=q(z)$; and (ii) $(\overline{X}_1, \Delta_1) $ is an snc pair at $x_1$ (and not just nc); see \cite[\S 2.4]{deJong96} or \cite{Conrad}. 
By construction, over an \'{e}tale neighbourhood of $q(z)$, the map $\pi\colon  (\overline{X}_1, \Delta_1, x_1) \to (\toric^{(n)}, \Delta^{(n)}, q(z))$ is locally \'{e}tale isomorphic to a toric birational morphism $ (Y,\Delta_Y, y) \to (Z^{n},\Delta^n, z)$. 

Now we want to modify $(\overline{X}_1 \supset \Delta_1)$ to a dlt pair, away from $x_1$. To this end, take a log resolution $\pi_1\colon  \overline{X}_2 \to \overline{X}_1$ of $\Delta_1$ which is an isomorphism over a neighbourhood of $x_1$. A relative MMP over $\overline{X}_1$ as in \cref{existence of dltcomp} ends with a dlt logCY compactification $(\overline{X}_3, \Delta_3)$ of  $M_{\mathrm{B}}(\mathrm{E},\Gl)$, together with a 0-dimensional stratum  $x_3\in \Delta_3$, such that the isomorphism $\overline{X}_3 \setminus \Delta_3 \simeq (Z \setminus \Delta)^{(n)}$ extends to locally \'{e}tale isomorphisms 
\[(\overline{X}_3, \Delta_3, x_3) \simeq (\overline{X}_1, \Delta_1, x_1) \simeq (Y, \Delta_Y, y).\] We summarise the construction in a diagram:
\begin{center}
    \begin{tikzcd}[scale=1.1]
   & (Y, \Delta_Y, y) \arrow{r}{bir.} \arrow[d, "\simeq"', "\text{loc. \'{e}t. at }y"]& (Z^{n}, \Delta^n, z) \arrow{d}{q}\\
   \overline{X}_2 \arrow{r}{\pi_1}  \arrow[rd, "MMP"', dashrightarrow]& (\overline{X}_1, \Delta_1, x_1) \arrow{r}{\pi}&(Z^{(n)}, \Delta^{(n)}, q(z))\\
  & (\overline{X}_3, \Delta_3, x_3). \arrow[u, "\simeq", "\text{loc. at }x_3"'] \arrow[ru, "bir."']&&
    \end{tikzcd}
\end{center}

\item[Step 2.] We construct a distinguished evaluation map $ev$ for the dlt pair $(\overline{X}_3, \Delta_3)$.

Let $\mu\colon  Y \to \R^{2n}$ be a symplectic moment map of $Y$; see \cite[Eq. 12.2.4]{CoxLittleSchenck2011}.
As $Y$ is smooth at $y$, up to a rigid motion, we can suppose that the image of a neighbourhood of $y$ via $\mu =(\mu_1, \ldots, \mu_{2n})$ is a neighbourhood of the origin of the standard orthant in $\R^{2n}$; see Delzant property in \cite[ p. 576]{CoxLittleSchenck2011}. Let $\Delta_{3,j}$ be the irreducible component of $\Delta_3$ which is locally given at $x_3$ by the equation $\{\mu_j=0\}$ (here we implicitly use the local isomorphism $(\overline{X}_3, \Delta_3, x_3) \simeq (Y, \Delta_Y, y)$). 

To construct the evaluation map $ev$, we use the compactification $(\overline{X}_3, \Delta_3)$ and a partition of unity as in \cref{lem:nonemptyassumption}, where in addition we require that the rug functions $\beta_j$ relative to the divisor $\Delta_{3,j}$ equals $\mu_j$ in a neighbourhood of $x_3$.
We can also set $\partial N_{\mathrm{B}}$ to be given, locally at $x_3$, by the equation
$\big\{\prod^{2n}_{j=1} \mu_j = \epsilon\big\}$,
namely the inverse image under $\mu$ of a real smoothing of the corner at the origin of the standard orthant. This is a generalisation of \cite[p. 7 (p. 231)]{Mumford61} in high dimension.

The upshot is that the fibres of $ev$ coincide near $x_3$ with the smooth fibres of the moment map $\mu$.
\item[Step 3.] Up to a diffeomorphism of the Hitchin base, the map $\chi \circ \Psi^{-1}\colon  (\C^*\times\C^*)^{(n)} \to \C^n$ is the $\mathfrak{S}_n$-quotient of the $\Log$ map
\[(\C^* \times \C^*)^{n} \to \R^{2n}, \qquad \Log\colon (z_1, \ldots , z_{2n}) \mapsto \bigg(-\log|z_1|, \ldots, -\log|z_{2n}|\bigg),\]
by \cite[\S 7.4]{GoldmanXia2008}.
Again up to a diffeomorphism of the base, the map $\Log$ can be identified with the restriction of $\mu\colon  Y \to \R^{2n}$ to the big torus of $Y$ by the explicit description of $\mu$ in \cite[Eq. 12.2.4]{CoxLittleSchenck2011}. 
\item[Step 4.] Finally, the fibres of $ev$ coincide near $x_3$ with the smooth fibres of the moment map $\mu$, or equivalently of the $\Log$ map, and so of $\chi \circ \Psi^{-1}$. In other words, we proved that we can choose $ev$ and $\chi \circ \Psi^{-1}$ such that they coincide on an open neighbourhood of $\mathbb{S}^{2n-1}$. By the proof of \cite[\S 7, Lemma 4]{Milnor1997}, the two maps are homotopic.
\end{enumerate}
\end{proof}

\section{Cohomological vs geometric P=W conjecture}\label{sec:cohomolovsgeometric}
In this section we compare the geometric P=W with the cohomological P=W conjecture of de Cataldo, Hausel and Migliorini \cite{CataldoHauselMigliorini2012} in arbitrary rank and genus. 

\subsection{The cohomological P=W conjecture}\label{sec:cohomologicalP=Wconjecture}
\spa{We first recall the definition of the perverse and weight filtrations. Denote by $IH^*(X)$ the intersection cohomology of a complex variety $X$ with middle perversity and rational coefficients. It is the hypercohomology of the intersection cohomology complex $\mathcal{IC}_X$, i.e.\ $IH^*(X) = \mathbb{H}^*(X, \mathcal{IC}_X)$; see \cite{GM83}. Recall that $IH^*(X)$ carries a canonical mixed Hodge structure, and so a \textbf{weight filtration} 
$$ W_k IH^*(X)\subseteq IH^*(X).$$
}

\spa{Now let $f\colon  X \to Y$ be a projective morphism of algebraic varieties, with $Y\subset \C^m$ affine and $\dim X = 2 \dim Y$. Then the intersection cohomology of $X$ is endowed with the \textbf{perverse filtration} associated to $f$, i.e.\
\begin{equation}\label{def:perversefiltration}
    P^{f}_k IH^d(X)\coloneqq \mathrm{Ker}\left\lbrace IH^d(X)\rightarrow \mathbb{H}^d(f^{-1}(\Lambda^{d-k-1}), \mathcal{IC}_X|_{f^{-1}(\Lambda^{d-k-1})}) \subseteq IH^d(X)
\right \rbrace,
\end{equation}
where $\Lambda^s\subset Y$ is a general $s$-dimensional linear section of $Y\subset \C^m$; see \cite[Theorem 4.1.1]{deCataldoMigliorini2010}. 
}

 Denoting by $\chi \colon M_{\mathrm{Dol}} \to \C^{N/2}$ the Hitchin map as in  \eqref{eq:Hitchinfibration}, we can now state the cohomological P=W conjecture.
\begin{conjecture}[Cohomological P=W conjecture]\label{conj:cohomP=W} For any $k \in \N$, we have
\[P^{\chi}_k IH^*(M_{\mathrm{Dol}})= \Psi^* W_{2k}IH^*(M_{\mathrm{B}})=\Psi^* W_{2k+1}IH^*(M_{\mathrm{B}}).\]
\end{conjecture}
In particular, \cref{conj:cohomP=W} implies the following \cref{conj:P=Wtopdegree}.
\begin{conjecture}[Cohomological P=W conjecture at the highest weight]\label{conj:P=Wtopdegree} 
For $N\coloneqq \dim M_{\mathrm{Dol}}$
\begin{align*}
    P^{\chi}_{N-1} IH^*(M_{\mathrm{Dol}}) & = \Psi^* W_{2N-1}IH^*(M_{\mathrm{B}}) \\
\Gr^P_{N}IH^*(M_{\mathrm{Dol}}) & \simeq \Gr^W _{2N}IH^*(M_{\mathrm{B}}).
\end{align*}
\end{conjecture}

We refer the interested reader to \cite{CataldoHauselMigliorini2012, deCataldoMaulikShen2019} and to the excellent survey~\cite{Migliorini} for twisted smooth character varieties; to \cite[Question 4.1.7]{deCataldoMaulik2018}, \cite{FelisettiMauri2020} and \cite{Mauri20} for the untwisted and singular case, namely the setting of this paper.

\subsection{An identification of Lagrangian tori}\label{sec:p=wgeomcohom} 
In order to relate the cohomological P=W conjecture to the geometric one, we need the definition of profound torus as in \cite{Harder2019}.

\begin{definition}\label{def:profoundtorus}
Let $x$ be a zero-dimensional stratum of a dlt pair $(X, \Delta)$. Choose local coordinates $z_1, \ldots, z_N$ centered at $x$ with $\Delta=\{ z_1 \cdot \ldots \cdot z_N =0\}$. For fixed radii $0 < r_j \ll 1$, $j=1, \ldots, N$, a \textbf{profound torus} $\mathbb{T}_x$ is 
\[\mathbb{T}_x = \{ (r_1 e^{i \theta_1}, \ldots, r_N e^{i \theta_N})\colon  \theta_1, \ldots, \theta_N \in [0,2\pi)\} \subset X.\]
\end{definition}
\begin{remark}
The ambient-isotopy type of $\mathbb{T}_x \subset X$ does not depend on the choice of the coordinates: $\mathbb{T}_x$ is a torus whose fundamental group is generated by $N$ loops $\gamma_j$, with $j=1, \ldots, N$, each of them winding around the divisor $\{z_j=0\}$; or equivalently, $\mathbb{T}_x$ is homotopic to a deleted neighbourhood of $x$.
\end{remark}
\spa{The geometric P=W conjecture claims the existence of a homotopy
\begin{equation}\label{eq:homotopy}
    \Theta \colon N^*_{\mathrm{Dol}}\times [0,1] \to \mathcal{D}(\partial M_{\mathrm{B}})
\end{equation}
between the maps $r \circ \chi$ and $ev \circ \Psi$.} 

\begin{Assumption}\label{assump:genericsmoothness}
Let $\mathrm{Sing} M_{\mathrm{Dol}}$ be the singular locus of $M_{\mathrm{Dol}}$. Then $\Theta((N_{\mathrm{Dol}}^*\cap \mathrm{Sing} M_{\mathrm{Dol}})\times [0,1])$ is strictly contained in $\mathcal{D}(\partial M_{\mathrm{B}})$.
\end{Assumption}

\begin{theorem}[Theorem A
]\label{thm:cohomologicalvsgeometric}
Under \cref{assump:genericsmoothness}, the geometric P=W conjecture implies the cohomological P=W conjecture at the highest weight.
\end{theorem}

\begin{proof}
Consider the homotopy $\Theta$ in \eqref{eq:homotopy}. By \cref{assump:genericsmoothness} and the density of smooth functions, there exists an open set $V$ in the locus where $\mathcal{D}( \partial M_{\mathrm{B}})$ is a manifold such that over $V$ we can perturbe $\Theta$ (and $r$)  to a smooth homotopy between $r \circ \chi$ and $ev \circ \Psi$. 

Let $z$ be a $\Theta$-regular value in $V$. Then $\Theta^{-1}(z)$ is a cobordism between $(r \circ \chi)^{-1}(z)$ and $(ev \circ \Psi)^{-1}(z)$. We first describe the topology of these fibres:
\begin{enumerate}
    \item A general fibre of the Hitchin map $M_{\mathrm{Dol}} \to \Lambda$ is an N-dimensional torus $\mathbb{T}_{\mathrm{Dol}}$. Hence, the general fibre of $\chi \colon N^*_{\mathrm{Dol}} \to \mathbb{S}^{N-1}$ is fibred in tori $\mathbb{T}_{\mathrm{Dol}}$ and homeomorphic to $(0,1) \times \mathbb{T}_{\mathrm{Dol}}$. Hence, there exists a homeomorphism
\[(r \circ \chi)^{-1}(z) \simeq \bigsqcup_{w \in r^{-1}(z)} (0,1) \times \mathbb{T}_{\mathrm{Dol}}.\]
\item Without loss of generality we can suppose that $z$ lies in the interior of a maximal cell of $\mathcal{D}(\partial M_{\mathrm{B}})$ corresponding to the 0-dimensional stratum $\Delta_J= \bigcap_{j \in J}\Delta_j$. Assume also that the evaluation map $ev$ is constructed by means of the partition of unity in \cref{lem:nonemptyassumption}, where in particular we require that the rug functions $\beta_j$ equals $|f_j|^2$ for a local equation $f_j$ of $\Delta_j$ at $\Delta_J$. By \eqref{eq:evexplicit}, the space $ev^{-1}(z)$ is fibred in profound tori around $\Delta_J$
\[\mathbb{T}_{\mathrm{B}}=\{x \in M_{\mathrm{B}} \colon |f_j(x)|=r_j\}\]
for some real numbers $r_j>0$. Hence, there exist homeomorphisms
\[(ev \circ \Psi)^{-1}(z) \simeq ev^{-1}(z) \simeq (0,1) \times \mathbb{T}_{\mathrm{B}}.\]
\end{enumerate}
In particular, we obtain that 
\[H^{N}(\chi^{-1}(w))\simeq H^N(\mathbb{T}_{\mathrm{Dol}})\simeq \Q \qquad H^{N}((ev \circ \Psi)^{-1}(z))\simeq H^N(\mathbb{T}_{\mathrm{B}})\simeq \Q,\]
with $w \in r^{-1}(z)$. By Stokes' theorem, there exists an isomorphism $\mathfrak{a}\colon  H^{N}((ev \circ \Psi)^{-1}(z)) \to H^{N}(\chi^{-1}(w))$ that fits into the following commutative triangle 
  \begin{equation}\label{eq:triangle}
    \begin{tikzcd}[scale=1.1]
    H^{N}(\Theta^{-1}(z)) \arrow{r}\arrow{rd} &  H^{N}(\chi^{-1}(w)) \simeq H^N(\mathbb{T}_{\mathrm{Dol}})\simeq \Q \\
    & H^{N}((ev \circ \varPsi)^{-1}(z))\simeq H^N(\mathbb{T}_{\mathrm{B}})\simeq \Q,\arrow[u,"\simeq","\mathfrak{a}"']
    \end{tikzcd}
\end{equation}  
and the other maps are the natural restriction maps. Indeed, for any closed form $\omega \in H^{N}(\Theta^{-1}(z), \R)$ we have
\begin{align*}
    0 =\int_{\Theta^{-1}(z)}d \omega & = \int_{(r \circ \chi)^{-1}(z)}\omega - \int_{(ev \circ \varPsi)^{-1}(z)}\omega \\
    & = \deg(r) \int_{ \chi^{-1}(w)}\omega - \int_{(ev \circ \varPsi)^{-1}(z)}\omega=\pm \int_{ \chi^{-1}(w)}\omega - \int_{(ev \circ \varPsi)^{-1}(z)}\omega.
\end{align*}
This means that there exists a morphism $\mathfrak{a}$ which makes the triangle \eqref{eq:triangle} commutative. The restriction  \[\mathrm{rest}_{\mathrm{Dol}} \colon H^N(M_{\mathrm{Dol}}) \simeq H^N(M_{\mathrm{Dol}} \times [0,1]) \to H^{N}(\Theta^{-1}(z)) \to  H^{N}(\chi^{-1}(w))\simeq H^N(\mathbb{T}_{\mathrm{Dol}})\] 
is surjective since the restriction of the top power of a $\chi$-relative ample class spans $H^N(\mathbb{T}_{\mathrm{Dol}}) \simeq \Q$. This implies that $\mathfrak{a}\ \neq 0$, i.e. it is an isomorphism.

Consider now the following commutative diagram
\begin{equation*}
   \begin{tikzcd}[scale=1.1]
   \chi^{-1}(w) \arrow[hookrightarrow, r] \arrow[hookrightarrow, d] & N^*_{\mathrm{Dol}} \times \{0\} \arrow[hookrightarrow, d] \arrow[r, "\chi"] & \mathbb{S}^{N-1} \arrow[d, "r"]\\
   \Theta^{-1}(z) \arrow[hookrightarrow, r] &  N^*_{\mathrm{Dol}} \times [0,1]  \arrow[r, "\Theta"] & \mathcal{D}(\partial M_{\mathrm{B}})\\
   (ev \circ \Psi)^{-1}(z) \arrow[hookrightarrow, u]\arrow[hookrightarrow, r] \arrow[ "\Psi","\simeq"', d] & N^*_{\mathrm{Dol}} \times \{1\} \arrow[hookrightarrow, u] \arrow[d, "\Psi","\simeq"'] & \\
    ev^{-1}(z)\arrow[hookrightarrow, r] & N^*_{\mathrm{B}}. \arrow[uur, "ev"'] & 
    \end{tikzcd}
\end{equation*}
Taking intersection cohomology, we obtain the following commutative diagram 
\begin{equation*}
   \begin{tikzcd}[scale=1.1]
     IH^{N}(M_{\mathrm{Dol}}) \arrow{r} & IH^{N}(N^*_{\mathrm{Dol}} \times\{0\}) \arrow{r} &   IH^{N}(\chi^{-1}(w))=H^{N}(\chi^{-1}(w))\simeq H^{N}(\mathbb{T}_{\mathrm{Dol}}) \\
     IH^{N}(M_{\mathrm{Dol}}) \arrow[u, "\mathrm{Id}"',"\simeq"] \arrow{r} & IH^{N}(N^*_{\mathrm{Dol}} \times\{1\}) \arrow[u, "\mathrm{Id}"',"\simeq"] \arrow{r} &  IH^{N}((ev \circ \varPsi) ^{-1}(z))=H^{N}((ev \circ \varPsi)^{-1}(w)) \arrow[u,"\simeq","\mathfrak{a}"']\\
   IH^{N}(M_{\mathrm{B}}) \arrow[u, "\Psi^*"',"\simeq"]\arrow{r} & IH^{N}(N^*_{\mathrm{B}}) \arrow[u, "\Psi^*"',"\simeq"]\arrow{r} & IH^{N}(ev ^{-1}(z))=H^{N}(ev ^{-1}(z))\simeq H^{N}(\mathbb{T}_{\mathrm{B}})\arrow[u,"\simeq"," \Psi^*"'].
    \end{tikzcd}
\end{equation*}
The four lemma gives the short exact sequences
\begin{equation}\label{diagram:P=Wtopdegree}
   \begin{tikzcd}[scale=1.1]
    \ker(\mathrm{rest}_{\mathrm{Dol}}) \arrow[hookrightarrow, r]  & IH^{N}(M_{\mathrm{Dol}}) \arrow{r}{\mathrm{rest}_{\mathrm{Dol}}} &  H^{N}(\mathbb{T}_{\mathrm{Dol}}) \\
   \ker(\mathrm{rest}_{\mathrm{B}}) \arrow[u,"\Psi^*"', "\simeq"]\arrow[hookrightarrow, r] & IH^{N}(M_{\mathrm{B}}) \arrow[u, "\Psi^*"',"\simeq"]\arrow{r}{\mathrm{rest}_{\mathrm{B}}} & H^{N}(\mathbb{T}_{\mathrm{B}})\arrow[u,"\simeq","\mathfrak{a} \circ \Psi^*"'],
    \end{tikzcd}
\end{equation}
where $\mathrm{rest}_{\mathrm{Dol}}$ and $\mathrm{rest}_{\mathrm{B}}$ are the natural restriction maps. 
 We show then that
\begin{align*}
 P^{\chi}_{N-1} IH^{N}(M_{\mathrm{Dol}}) 
 \stackrel{(i)}{=} \ker(\mathrm{rest}_{\mathrm{Dol}}) 
 & \stackrel{\cref{diagram:P=Wtopdegree}}{=} \Psi^*(\ker(\mathrm{rest}_{\mathrm{B}})) 
 \stackrel{(iii)}=  \Psi^*(W_{2N-1} IH^{N}(M_{\mathrm{B}})) \\
 \Gr^P_{N}IH^*(M_{\mathrm{Dol}}) 
 \stackrel{(ii)}{\simeq} H^{N}(\mathbb{T}_{\mathrm{Dol}})
 & \stackrel{\cref{diagram:P=Wtopdegree}}{\simeq} H^{N}(\mathbb{T}_{\mathrm{B}})
 \stackrel{(iv)}{\simeq} \Gr^W_{2N}IH^*(M_{\mathrm{B}}).
\end{align*}
Indeed, $(i)$ holds by \eqref{def:perversefiltration};
$(ii)$ follows from 
the surjectivity of $\mathrm{rest}_{\mathrm{Dol}}$. 
In order to prove $(iii)$ and $(iv)$ we proceed as follows. 
By \cite[Theorem 2.11]{Harder2019}\footnote{All the compactifications in \cite{Harder2019} are snc. Therefore, we apply \cite[Theorem 2.11]{Harder2019} to a log resolution of a dlt compactification $(\overline{M}_{\mathrm{B}}, \partial M_{\mathrm{B}})$ which is an isomorphism over the snc locus of the pair, in particular over a neighbourhood of the zero dimensional lc centres of $(\overline{M}_{\mathrm{B}}, \partial M_{\mathrm{B}})$. We descend the result to $M_{\mathrm{B}}$ using for instance \cite[Proposition 4.3]{Payne}.} (which relies on \cite{EN02}) the restriction $\mathrm{rest}_{\mathrm{B}}$ factors as follows
\begin{equation}\label{eq:restB}
    \mathrm{rest}_{\mathrm{B}}\colon  IH^{N}(M_{\mathrm{B}}) \to \Gr^W _{2N}IH^{N}(M_{\mathrm{B}}) \xrightarrow{\mathfrak{b}} H^{N}(\mathbb{T}_{\mathrm{B}}) \simeq \Q.
\end{equation} By the commutativity of \eqref{diagram:P=Wtopdegree},  the surjectivity of $\mathrm{rest}_{\mathrm{Dol}}$ implies the surjectivity of $\mathrm{rest}_{\mathrm{B}}$ (the vertical arrows are isomorphisms), which gives the surjectivity of $\mathfrak{b}$ by \eqref{eq:restB}.  Further, the domain of $\mathfrak{b}$, i.e.\ $\Gr^W _{2N}IH^{N}(M_{\mathrm{B}})$, has rank one, because
\begin{equation}\label{eq:GRcompare}
    \dim \Gr^W _{2N}IH^{N}(M_{\mathrm{B}})=\dim \widetilde{H}^{N-1}(\mathcal{D}(\partial M_{\mathrm{B}}))=\dim \widetilde{H}^{N-1}(\mathbb{S}^{N-1})=1
\end{equation}
by \cref{lem:dualcomplexcohomology} and the geometric P=W conjecture. All together, this gives that $\mathfrak{b}$ is an isomorphism, i.e. $(iv)$ holds, and so $(iii)$ does too.
\end{proof}

\begin{remark}
We do not expect the geometric P=W conjecture to inform us about the cohomological P=W conjecture in lower weight. Note that the homotopy class of the map $ev \circ \Psi$ records information only about the general fibre. On the other hand, to understand the perverse filtration in lower perversity it is necessary by \cref{def:perversefiltration} to know the topology of the preimage of lower codimensional affine subspaces of the Hitchin base. To this extent, \cref{thm:cohomologicalvsgeometric} is optimal, and a geometric explanation of the cohomological P=W conjecture in all weights requires a refinement of \cref{conjKNPS}.
\end{remark}

\begin{remark}[Relation to other work I]
Harder explains in \cite{Harder2019} how \eqref{eq:GRcompare} and $(iv)$ should be seen as evidence of the conjectural existence of logCY compactifications of character varieties. Analogues of \cref{thm:cohomologicalvsgeometric} have been proved in the surface case in \cite[\S 5]{Szabo2018} and \cite[\S 4]{Harder2019}.
\end{remark}

\begin{remark}[Relation to other work II]\label{Rmk:relationotherworksII}
In \cite[Proposition 5.7]{EvansMauri2019}, Evans and the first author constructed evaluation maps $ev$ whose general fibres are Lagrangian with respect to a symplectic form on $\overline{M}_{\mathrm{B}}$: this additional structure on $ev$ may be useful to show the homotopy commutativity of the diagram \eqref{diagP=W}, as follows. The smooth locus $M^{\mathrm{sm}} \coloneqq M^{\mathrm{sm}}_{\mathrm{B}}$ (as a $\mathcal{C}^{\infty}$-manifold) carries a canonical hyper-K\"{a}hler structure $(g,I,J,K)$, where the complex manifolds $(M^{\mathrm{sm}}, I)$ and $(M^{\mathrm{sm}}, J)$ are biholomorphic to $M^{\mathrm{sm}}_{\mathrm{Dol}}$ and $M^{\mathrm{sm}}_{\mathrm{B}}$ respectively. If the general fibre of $ev$ was chosen special Lagrangian with respect to the symplectic form $g(J \cdot, \cdot)$ and the $J$-holomorphic volume form, then $ev^{-1}(z)$ would be holomorphic with respect to the complex structure $I$ by \cite[\S 3]{Hitchin1997}, and since the Hitchin base is affine, $ev^{-1}(z)$ should generically coincide with a fibre of the Hitchin fibration. Finally, Step 4 of \cref{thm:P=Wconjecturegeomgenusonetext} should give the required commutativity of \eqref{diagP=W}. 

This correspondence between the Lagrangian tori is also the intuition behind \cref{assump:genericsmoothness}: the profound torus $ev^{-1}(z)$ should be ambient-isotopic to a special Lagrangian torus $\chi^{-1}(w)$, and the isotopy should be local (alias fixing a complement of the neighbourhood of a zero dimensional stratum of $\partial M_{\mathrm{B}}$, and thus occurring away from $\mathrm{Sing} M_{\mathrm{Dol}}$). This is compatible with the expectations in \cite[Conjecture 7.3]{Auroux07}. See also \cite{YangLi} where a similar result for Calabi--Yau degenerations has been achieved, conditional to a conjecture in non-Archimedean geometry.
\end{remark}

\section{Dual complex of character varieties}\label{sec:dualcomplexarbgenus}
In this section we collect some results about the rational homology and the fundamental group of dual complex of character varieties. We start with \cref{lem:dualcomplexcohomology}, which  generalises \cite[Theorem 4.4]{Payne} to the dlt setting.
\begin{lemma}\label{lem:dualcomplexcohomology}
Let \((\overline{X}, \partial X)\) be a dlt compactification of a variety \(X\) of dimension $N$.
For all non negative integers \(i\), the reduced cohomology group $\widetilde{H}^{i-1}(\mathcal{D}(\partial X))$ is isomorphic to 
$\Gr^W_{2N}H^{2N-i}(X) \simeq \Gr^W_{2N}IH^{2N-i}(X)$.
\end{lemma}
\begin{proof}
Let \(\pi\colon  (\overline{Y}, \partial Y) \to (\overline{X}, \partial X)\) be a log resolution extending a resolution $Y\coloneqq \pi^{-1}(X) \to X$. By \cite[Proposition 4.3, Theorem 4.4]{Payne}, there exist isomorphisms
\[
\widetilde{H}^{i-1}(\mathcal{D}(\partial Y)) \simeq \Gr^W_{2N}H^{2N-i}(Y) \simeq \Gr^W_{2N}H^{2N-i}(X).
\]
By \cite[Theorem 2.2.3.(a)]{deCataldoMigliorini2010}, there are morphisms of mixed Hodge structures $H^*(X) \to IH^*(X) \hookrightarrow H^*(Y)$ which yield
\[
\widetilde{H}^{i-1}(\mathcal{D}(\partial Y))\simeq \Gr^W_{2N}H^{2N-i}(X) \simeq \Gr^W_{2N}IH^{2N-i}(X) \simeq \Gr^W_{2N}H^{2N-i}(Y).
\]
Further, the dual complex \(\mathcal{D}(\partial Y)\) is homotopy equivalent to \(\mathcal{D}(\partial X)\), due to \cite[Theorem 28.(2)]{deFernexKollarXu2012}. This implies \(\widetilde{H}^{i-1}(\mathcal{D}(\partial Y))\simeq \widetilde{H}^{i-1}(\mathcal{D}(\partial X))\), and this concludes the proof.
\end{proof}
\begin{cor}[Rational homology]\label{cor:rational homology}
Let $C$ be a compact Riemann surface, and $G$ be either $\mathrm{GL}_2$ or $\mathrm{SL}_2$. Set $N \coloneqq \dim M_{\mathrm{B}}(C,G)$. Then $\mathcal{D} (\partial M_{\mathrm{B}}(C,G))$ has the rational homology of a sphere of dimension $N-1$.
\end{cor}
\begin{proof}
By \cite[\S 6]{Payne} $\mathcal{D}( \partial M_{\mathrm{B}}(C,G))$ has the rational homology of a wedge of spheres of dimension $N-1$. By Poincar\'{e} duality, the dimension of $\Gr^W_{2N}IH^{2N-i}( M_{\mathrm{B}}(C,G))$ is the constant term of the intersection E-polynomial in \cite[Theorem 1.4]{Mauri20}, alias one. We conclude then by \cref{lem:dualcomplexcohomology}.
\end{proof}

\begin{theorem}[Simply-connectedness]\label{thm:simplyconnectedness}
Let $C$ be a compact Riemann surface of genus $g$, and $G$ be either $\mathrm{GL}_n$ or $\mathrm{SL}_n$. If $\dim M_{\mathrm{B}}(C,G)>2$, then $\mathcal{D}( \partial M_{\mathrm{B}}(C,G))$ is simply-connected.
\end{theorem}
\begin{proof}
For some affine embedding ${M}_{\mathrm{B}}(C,G) \subseteq \C^m$, let $\overline{X}_1$ be the (normalisation of) the projective closure of ${M}_{\mathrm{B}}(C,G)$ in $\PP^m$. By \cref{existence of dltcomp}, a birational modification of $\overline{X}_1$ is a dlt compactification $(\overline{M}_{\mathrm{B}}, \partial M_{\mathrm{B}})$ of $M_{\mathrm{B}}=M_{\mathrm{B}}(C,G)$. Let $\overline{M}^{\mathrm{sm}}_{\mathrm{B}}$ be its smooth locus. By \cite[Corollary 40 and (35.2)]{KollarXu} (and since $\dim M_{\mathrm{B}}(C,G)>2$) there exists a surjection
\begin{equation}\label{eq:fundvarietydualcompalex}
    \pi_1(\overline{M}^{\mathrm{sm}}_{\mathrm{B}}) \twoheadrightarrow \pi_1(\mathcal{D}(\partial M_{\mathrm{B}})).
\end{equation}

If $G = \Sl$, $g>1$ and $(g,n)\neq (2,2)$, then $M^{\mathrm{sm}}_{\mathrm{B}}(C,\Sl)$ is simply-connected by \cite[Proposition 5.30]{FelisettiMauri2020}. Since the inclusion of a dense Zariski open subset $U$ in a variety $X$ yields the surjection $\pi_1(U) \twoheadrightarrow \pi_1(X)$, we obtain 
\[1=\pi_1(M^{\mathrm{sm}}_{\mathrm{B}}(C, \Sl)) \twoheadrightarrow \pi_1(\overline{M}^{\mathrm{sm}}_{\mathrm{B}}(C, \Sl)) \twoheadrightarrow \pi_1(\mathcal{D}( \partial M_{\mathrm{B}}(C,\Sl))).\]

For $G=\Gl$, $g>1$ and $(g,n)\neq (2,2)$, we construct the dlt compactification $\overline{M}_{\mathrm{B}}$ as follows. The algebraic torus $\mathbb{K} \simeq (\C^{*})^{2g}$ acts on $M_{\mathrm{B}}(C, \Gl)$ by
\begin{align*}
    \tau\colon  \mathbb{K} \times M_{\mathrm{B}}(C, \Gl) & \to M_{\mathrm{B}}(C, \Gl), \\
    ((\lambda_1, \ldots, \lambda_{2g}), (A_1, B_1, \ldots, A_g, B_g)) & \mapsto (\lambda_1 A_1, \lambda_2 B_1,  \ldots, \lambda_{2g-1}A_g,  \lambda_{2g}B_g)
\end{align*}
 and its categorical quotient is the character variety $M_{\mathrm{B}}(C, \PGl)$. Let $\overline{M}_{\mathrm{B}}(C, \PGl)$ be a projective compactification of $M_{\mathrm{B}}(C, \PGl)$, and choose $\overline{X}'_1$ a resolution of the indeterminacy of the rational map $\overline{X}_1 \dashrightarrow \overline{M}_{\mathrm{B}}(C, \PGl)$ extending the quotient map $q\colon  M_{\mathrm{B}}(C, \Gl) \to M_{\mathrm{B}}(C, \PGl)$, i.e.
\begin{center}
    \begin{tikzcd}[scale=1.1]
    &(\overline{X}'_1, \Delta'_1 \coloneqq (f^{'})^{-1}(\Delta_1))\arrow[dl, "f'"'] \arrow[dr, "g"]&\\
    (\overline{X}_1, \Delta_1) \arrow[rr, dashed, "q"]& &\overline{M}_{\mathrm{B}}(C, \PGl).
    \end{tikzcd}
\end{center}
Note that $f'$ is an isomorphism over $M_{\mathrm{B}}(C, \Gl)$. Up to blowups centered over $\Delta'_1$, we can also suppose that the general $g$-fibre is an snc compactification $(\overline{\mathbb{K}}, \Delta_{\mathbb{K}})$ of $\mathbb{K}$. Finally, as in \cref{existence of dltcomp}, we can construct a dlt modification $(\overline{M}_{\mathrm{B}}(C, \Gl), \partial M_{\mathrm{B}}(C, \Gl)) \to (\overline{X}'_1, \Delta'_1)$  which is an isomorphism over $M_{\mathrm{B}}(C, \Gl)$ and over the general $g$-fibre in $\overline{X}'_1$; cf \cite[Corollary 1.36]{KollarKovacs}. Hence, we obtain a morphism $\overline{M}_{\mathrm{B}}(C, \Gl) \to \overline{M}_{\mathrm{B}}(C, \PGl)$ whose general fibre is again $(\overline{\mathbb{K}}, \Delta_{\mathbb{K}})$.

 The determinant $\det\colon  M^{\rm sm}_{\mathrm{B}}(C, \Gl) \to M_{\mathrm{B}}(C, \C^*)\simeq \mathbb{K}$, given by $\det(A_1, \ldots, B_g) = (\det A_1, \ldots, \det B_g)$, is an isotrivial fibration with simply-connected fibre $M^{\rm sm}_{\mathrm{B}}(C, \Sl)$. This yields an isomorphism \[\pi_1(M^{\mathrm{sm}}_{\mathrm{B}}(C, \Gl))\simeq \pi_1(M_{\mathrm{B}}(C, \C^*))\simeq \pi_1(\mathbb{K}).\]
Recall that $\overline{\mathbb{K}}$ is simply-connected, because $\overline{\mathbb{K}}$ is a smooth projective rational variety. Then the natural inclusions induce the following commutative square
\begin{equation}\label{eq:fundamentalgroups}
    \begin{tikzcd}[scale=1.1]
    \pi_1(M^{\mathrm{sm}}_{\mathrm{B}}(C, \Gl)) \ar[r, two heads]& \pi_1(\overline{M}^{\mathrm{sm}}_{\mathrm{B}}(C, \Gl))\\
    \pi_1(\mathbb{K})\ar[r,, two heads] \ar[u, "\simeq"]& \pi_1(\overline{\mathbb{K}})=1. \ar[u]
    \end{tikzcd}
\end{equation}
By \eqref{eq:fundamentalgroups} and \eqref{eq:fundvarietydualcompalex}, 
we conclude
$\pi_1(\mathcal{D} (\partial M_{\mathrm{B}}(C, \Gl))) \simeq \pi_1(\overline{M}^{\mathrm{sm}}_{\mathrm{B}}(C, \Gl))=1$.

Finally, if $g=1$ or $(g,n)=(2,2)$, $M_{\mathrm{B}}$ admits a (unique crepant) symplectic resolution $\widetilde{M}_{\mathrm{B}}\to {M}_{\mathrm{B}}$; see for instance \cite[\S 2.3-2.4]{BellamySchedler2019}. This means that a holomorphic symplectic form on $M^{\mathrm{sm}}_{\mathrm{B}}$ extends to a symplectic form on $\widetilde{M}_{\mathrm{B}}$. Now if $g=1$, then $\widetilde{M}_{\mathrm{B}}(C, \Sl)$ is the fibre over $(1,1)$ of the isotrivial multiplication map $(\C^*\times \C^*)^{[n]} \to (\C^* \times \C^*)$, which can be easily seen to be simply-connected. If $(g,n)=(2,2)$, then $\pi_1({M}^{\mathrm{sm}}_{\mathrm{B}}(C, \SL_2))$ is generated by a single loop of order two (cf \cite[Proposition 5.30]{FelisettiMauri2020}), which is contracted in the resolution. In both cases we have $\pi_1(\widetilde{M}_{\mathrm{B}}(C, \Sl))=1$, and the previous arguments applied to $\widetilde{M}_{\mathrm{B}}$ gives $\pi_1(\mathcal{D}(\partial \widetilde{M}_{\mathrm{B}}))=1$ (both for $\Sl$ and $\Gl$). By \cite[Theorem 1]{deFernexKollarXu2012}, we conclude $\pi_1(\mathcal{D}(\partial M_{\mathrm{B}}))=\pi_1(\mathcal{D}(\partial \widetilde{M}_{\mathrm{B}}))=1$.
\end{proof}

\begin{remark}[Relation with other work III]\label{rmk:relationotherworksIII} It is reasonable to formulate the geometric P=W conjecture also for moduli spaces $X_{\mathrm{B}}=X_{\mathrm{B}}(C^{\circ},G, \mathcal{C}_k)$ of $G$-local systems on punctured curves $C^{\circ}\coloneqq C \setminus \{p_k\}_{k\in K}$, such that conjugacy classes $\mathcal{C}_k$ of the monodromy around the punctures $p_k$ are fixed. 
Actually, the first evidence of the geometric P=W conjecture have been provided in this setting: when $X_{\mathrm{B}}$ is the $\mathrm{SL}_2$-character variety of local systems on a punctured sphere 
(with fixed trace at the punctures),
Simpson proves 
in~\cite[Theorem 1.1]{Simpson2016} that the dual boundary complex $\mathcal{D}(\partial M_{\mathrm{B}})$ has the homotopy type of a sphere; see also~\cite[Theorem 1.4]{Komyo2015}.
For punctured spheres, there is no known proof of the commutativity of the diagram \eqref{diagP=W}, except for some partial results in the case of five punctures due to Szab\'{o}
\cite{Szabo2021}. The full geometric P=W conjecture has also been proved in the Painlev\'{e} cases in \cite{Szabo2018, Szabo2019, NemethiSzabo2020}. Therefore, the results contained in this paper are the first evidence of the geometric P=W conjecture for compact curves and high rank.
\end{remark}

\begin{remark}[Relation with other work IV]
Let $X_{\mathrm{B}}$ be the character variety defined in \cref{rmk:relationotherworksIII}. For generic choice of the monodromy of the local systems around the punctures, $X_{\mathrm{B}}$ is smooth, and the analogues of \cref{cor:rational homology} and \cref{thm:simplyconnectedness} for $X_{\mathrm{B}}$ and structural group $G=\Gl$ follow at once from \cite{Mellit2019}. Indeed, curious Hard Lefschetz \cite[Theorem 1.5.3]{Mellit2019} implies 
\[\widetilde{H}^{N-1}(\mathcal{D}(\partial X_{\mathrm{B}}))\simeq \Gr^W_{2N}H^{N}(X_{\mathrm{B}})\simeq H^0(X_{\mathrm{B}})\simeq \Q,\] 
which is an analogue of \cref{cor:rational homology} (but no intersection cohomology is needed). Moreover, an snc compactification $\overline{X}_{\mathrm{B}}$ of $X_{\mathrm{B}}$ is smooth, projective and stably rational by the existence of the cell decomposition \cite[Theorem 1.4.1.]{Mellit2019}, thus simply-connected, and so by \cref{eq:fundvarietydualcompalex} we deduce the simply-connectedness of $\mathcal{D} (\partial X_{\mathrm{B}})$.

The argument is more subtle in the singular case, as we have shown above in the proofs of \cref{cor:rational homology} and \cref{thm:simplyconnectedness}. Indeed, the results of \cite{Mellit2019} are no longer available for singular character varieties. Note also that there are examples of rational dlt (even logCY) compactifications of singular affine varieties whose dual boundary complex is not simply-connected; see \cite[Example 60]{KollarXu} or \cite[Example 7.4]{Mauri2020dual}. However, by \cref{cor:rational homology} and \cref{thm:simplyconnectedness} this does not happen for dlt compactifications of character varieties, compatibly with the expectation in \cite[\S 1.2]{Simpson2016}.
\end{remark}

It is worthy to remark that, both in the smooth and in the singular case, the most elusive property necessary to show that the dual boundary complex has the homotopy of a sphere seems the vanishing of the torsion of $H^*(\mathcal{D} (\partial X_{\mathrm{B}}), \Z)$ or $H^*(\mathcal{D}(\partial M_{\mathrm{B}}), \Z)$; cf \cite[bullet points right before \S 6]{KollarXu}.

\appendix
\section{Independence of the geometric P=W conjecture of auxiliary choices}\label{appendixindep}
In this appendix we check the independence of the geometric P=W conjecture of all the auxiliary choices and claims made in \cref{subsec:defnMBMDol}. 

\begin{lemma}\label{lem:inclusionhomotopyequivalence}
The inclusion $\Psi(N^*_{\mathrm{Dol}})\hookrightarrow N^*_{\mathrm{B}}$ is a homotopy equivalence.
\end{lemma}
\begin{proof}
Choose rug functions $\beta_{\mathrm{B}}\colon \overline{M}_{\mathrm{B}} \to \R$ and $\beta_{\mathrm{Dol}}\colon \overline{M}_{\mathrm{Dol}} \to \R$, and real numbers $\delta_{{\mathrm{B}}}$ and $\delta_{\mathrm{Dol}}$ such that $N_{\mathrm{B}} = \beta_{\mathrm{B}}^{-1}[0, \delta_{\mathrm{B}})$ and $N_{\mathrm{Dol}} = \beta_\mathrm{Dol}^{-1}[0, \delta_\mathrm{Dol})$. Observe that the homotopy types of $N_{\mathrm{B}}$ and $N_{\mathrm{Dol}}$ is independent of these choices by the uniqueness of semialgebraic neighbourhoods; see Proposition \ref{uniqueness}. 

In order to show that $\Psi(N^*_{\mathrm{Dol}}) \hookrightarrow N^*_{\mathrm{B}}$ is a homotopy equivalence, we argue as in the proof of \cite[Proposition 1.6]{D83}. Consider the system of neighbourhoods of $\partial M_{\mathrm{Dol}}$ and $\partial M_{\mathrm{B}}$ respectively
\[ N_{\mathrm{Dol}, k}= \beta_{\mathrm{Dol}}^{-1}\bigg[0, \frac{\delta_{\mathrm{Dol}}}{k}\bigg) \subseteq \overline{M}_{\mathrm{Dol}} \qquad N_{B,k} = \beta_{\mathrm{B}}^{-1}\bigg [0, \frac{\delta_{\mathrm{B}}}{k} \bigg) \subseteq \overline{M}_{\mathrm{B}}  \quad \mathrm{ with }\, k \in \mathbb{N}^*.\]
Choose $k_1, k_2\in \mathbb{N}$ such that
$
\Psi(N^*_{\mathrm{Dol}, k_1}) \subseteq N^*_{B, k_2} \subseteq \Psi(N^*_{\mathrm{Dol}}) \subseteq N^*_{B}$.
The inclusion $N^*_{B, k_2} \subseteq N^*_{B}$ is a homotopy equivalence by Proposition \ref{uniqueness}. The same holds for $\Psi(N^*_{\mathrm{Dol}, k_1}) \subseteq \Psi(N^*_{\mathrm{Dol}})$, since $N^*_{\mathrm{Dol}, k_1}$ and $N^*_{\mathrm{Dol}}$ are stratified isotopic by \cref{uniqueness}, and $\Psi$ is a homeomorphism. Therefore, we have the following sequence in homotopy
\[ \pi_i(\Psi(N^*_{\mathrm{Dol}, k_1})) \to \pi_i( N^*_{\mathrm{B}, k_2}) \to \pi_i(\Psi(N^*_{\mathrm{Dol}})) \to \pi_i(N^*_{\mathrm{B}}),
\]
which has the property that the composition of two consecutive morphisms is an isomorphism. Hence, the inclusion $N^*_{B, k_2} \subseteq \Psi(N^*_{\mathrm{Dol}})$ induces isomorphisms of homotopy groups, and it is a homotopy equivalence. We conclude that $\Psi(N^*_{\mathrm{Dol}}) \to N^*_{\mathrm{B}}$ is a homotopy equivalence.
\end{proof}

The geometric P=W conjecture is independent of the choice of the dlt compactification, the semialgebraic deleted neighbourhood \(N_{\mathrm{B}}^*\), \(N^*_i\), and of the partition of unity \(\{\varphi_i\}\), as shown in \cref{lem:indep2}. 
 
To this end, we first extend the construction of a retraction between dual complexes of snc pairs in \cite[\S 4.2]{BoucksomJonsson2017} to the $\Q$-factorial dlt setting. Let $(X, \Delta_X = \sum_{i \in I}\Delta_{X,i})$ be a dlt compactification of $U = X \setminus \Delta_X$ satisfying \cref{assumpdivisor}.
We say that a morphism of reduced dlt pairs $g \colon (Z, \Delta_Z = \sum_{r \in R} \Delta_{Z,r}) \to (Y, \Delta_Y = \sum_{k \in K}\Delta_{Y,k})$ over $(X, \Delta_X)$ is a \textbf{birational morphism of dlt compactifications over $(X, \Delta_X)$} if
\begin{itemize}
\item the structural morphisms $f_Z \colon Z \to X$ and $f_Y \colon Y \to X$ are birational, and $f_Z = f_Y \circ g$;
    \item $U \simeq Y \setminus \Delta_Y \simeq X \setminus \Delta_X$;
    \item $(Z,\Delta_Z)$ and $(Y,\Delta_Y)$ satisfy \cref{assumpdivisor}.
\end{itemize}

Suppose that $X$, $Y$ and $Z$ are $\Q$-factorial. 
Write $f_Y^*{\Delta_X}= \sum_{k \in K}n_k \Delta_{Y,k}$ for some $n_k \in \Q$, and 
$$f_Z^*{\Delta_X}
= g^*(f_Y^*{\Delta_X}) = \sum_{k \in K}n_k g^*(\Delta_{Y,k})
= \sum_{k \in K}n_k \bigg(\sum_{r \in R} a_{kr} \Delta_{Z,r} \bigg)
\eqqcolon \sum_{r \in R}m_r \Delta_{Z,r}$$ 
for some $a_{kr} \in \Q$ and $m_r= \sum_{k \in K} n_k a_{kr}$.

A stratum $\Delta_{Z,S}$ of $\Delta_Z$, with $S \subseteq R$, corresponds to a cell $\sigma_S$ of $\mathcal{D}(\Delta_Z)$, which we identify with the standard simplex $\{\sum_{s \in S} w_s =1\}\subset \R^S_{\geq 0} \subseteq \R^R$. Let $\Delta_{Y,L}$ be the minimum stratum of $\Delta_Y$ such that $g(\Delta_{Z,S}) \subseteq \Delta_{Y,L}$. Denote by $\sigma_L$ the corresponding cell in $\mathcal{D}(\Delta_Y)$, identified with the standard simplex $\{\sum_{l \in L} u_l =1\}\subset \R^L_{\geq 0} \subseteq \R^K$.

We define the map $r_{ZY} \colon \mathcal{D}(\Delta_Z) \to \mathcal{D}(\Delta_Y)$ on each cell $\sigma_S$ by
\begin{equation}\label{eq:defretraction}
(w_s)_{s \in S} \in \sigma_S \mapsto \bigg(n_l \sum_{s \in S} a_{ls} \frac{w_s}{m_s}\bigg)_{l \in L} \in \sigma_L.\end{equation}
Observe that $r_{ZX} = r_{YX} \circ r_{ZY}$, with $f_X= \textrm{Id}_X$. We show now that the map $r_{YX}$ is a retraction. To this end, we first recall the definition of simple blow-up.

\begin{definition}
Let $t\colon (V_2, \Delta_{V_2}\coloneqq (t^{-1}\Delta_{V_1})_{\textup{red}}) \to (V_1, \Delta_{V_1})$ is a morphism of snc pairs. We say that $\rho$ is a \textbf{simple blow-up} if it is a
blowup along a smooth, connected subvariety $W$ of $\Delta_{V_1}$ meeting transversely
(or not at all) every irreducible component of $\Delta_{V_1}$ that does not contain $W$; see for instance \emph{\cite[Definition 22]{KontsevichSoibelman}}.
\end{definition}

\begin{lemma}
The map $r_{YX}$ is a retraction.
\end{lemma}
\begin{proof}
The idea of the proof is to reduce to the snc case and apply \cite[Proposition 4.3]{BoucksomJonsson2017}. Let $X^{\mathrm{snc}}$ be the largest locus in $X$ where the pair $(X, \Delta_X)$ is snc. There exists a birational modification $g\colon Z \to Y$ such that $h \coloneqq f \circ g \colon Z \to X$ is a composition of simple blow-up over $X^{\mathrm{snc}}$; see for instance \cite[Lemma 4.1]{BoucksomJonsson2017}. Up to passing to a dlt modification of $Z$ 
which is an isomorphism over $h^{-1}(X^{\mathrm{snc}})$, we can suppose that $g\colon (Z, \Delta_Z) \to (Y, \Delta_Y)$ is a morphism of dlt compactifications. 
Note that $(X^{\mathrm{snc}}, \Delta^{\mathrm{snc}}_X\coloneqq \Delta_X|_{X^{\mathrm{snc}}})$ and $(h^{-1}(X^{\mathrm{snc}}), h^{-1}(\Delta^{\mathrm{snc}}_X))$ are snc pairs by definition of dlt pair and simple blow-up. 

Taking dual complexes, we have that $\mathcal{D}(\Delta_X) = \mathcal{D}(\Delta^{\mathrm{snc}}_X)$ again by definition of dlt pair, and $\mathcal{D}(h^{-1}(\Delta^{\mathrm{snc}}_X))$ is a subcomplex of $\mathcal{D}(\Delta_Z)$. By construction, the restriction $r_{ZX}$ to $\mathcal{D}(h^{-1}(\Delta^{\mathrm{snc}}_X))$ coincides with $r_{h^{-1}(X^{\mathrm{snc}}) X^{\mathrm{snc}}}$, which is a retraction by \cite[Lemma 4.7]{BoucksomJonsson2017}. The commutativity of the square
    \[
    \begin{tikzcd}
    \mathcal{D}(\Delta_Z) \arrow["r_{ZX}"]{rr} & \qquad &\mathcal{D}(\Delta_X)\\
    \mathcal{D}(h^{-1}(\Delta^{\mathrm{snc}}_X)) \arrow[hookrightarrow]{u} \arrow["r_{h^{-1}(X^{\mathrm{snc}}),X^{\mathrm{snc}}}"]{rr} & & \mathcal{D}(\Delta^{\mathrm{snc}}_X)\arrow["="]{u}.  
    \end{tikzcd}
\]
implies that $r_{ZX}$ is a retraction, and so $r_{YX}$ is a retraction too, since $r_{ZX}= r_{YX} \circ r_{ZY}$.
\end{proof}

\begin{lemma}[Independence of auxiliary choices]\label{lem:indep2} 
Let $(X, \Delta_X = \sum_{i \in I}\Delta_{X,i})$ and $(Y, \Delta_Y = \sum_{k \in K}\Delta_{Y,k})$ be dlt compactifications of $M_{\mathrm{B}}$. 
Let $ev_X\colon N^*_X \to \mathcal{D}(\Delta_X)$ and $ev_Y\colon N^*_Y \to \mathcal{D}(\Delta_Y)$ be evaluation maps for suitable semialgebraic neighbourhoods $N_X$ of $\Delta_X$ and $N_Y$ of $\Delta_Y$ containing $\Psi(N^*_{\mathrm{Dol}})$. If $\mathcal{D}(\Delta_X)$ (equivalently $\mathcal{D}(\Delta_Y)$) has the homotopy type of a sphere, then there exists a homotopy equivalence $\mathfrak{s}\colon \mathcal{D}(\Delta_Y) \to \mathcal{D}(\Delta_X)$ which makes the following diagram homotopy commutative:
\[
    \begin{tikzcd}
    & N^*_X \arrow{r}{ev_X}& \mathcal{D}(\Delta_X) \\
    N^*_{\mathrm{Dol}}\arrow{ru}{\Psi}\arrow["\Psi"']{rd}& &\\
    & N^*_Y \arrow{r}{ev_Y}& \mathcal{D}(\Delta_Y).\arrow["\mathfrak{s}"']{uu}
    \end{tikzcd}
\]
\end{lemma}

\begin{proof}
We proceed in several steps.
\begin{itemize}
    \item[Step 1] We can suppose that $(X, \Delta_X)$ and $(Y, \Delta_Y)$ are $\Q$-factorial. If $(X, \Delta_X)$ is not so (the same argument holds for $(Y, \Delta_Y)$ too), we take a small dlt $\Q$-factorial modification $\pi_1\colon (X_1, \Delta_1) \to (X, \Delta_X)$ such that $\Supp(\Delta_1)=\pi_1^{-1}(\Supp(\Delta_X))$ as in \cite[Definition 26]{deFernexKollarXu2012} (mind that in our case $\Delta_1$ and $\Delta_X$ are reduced). Since $M_{\mathrm{B}}$ has $\Q$-factorial singularities, $\pi_1$ is an isomorphism over $M_{\mathrm{B}}$, and so $(X_1, \Delta_1)$ is a $\Q$-factorial dlt compactification of $M_{\mathrm{B}}$. Note that $\mathcal{D}(\Delta_1)=\mathcal{D}(\Delta_X)$, and rug functions for $\Delta_{X, i}$ pull-back to rug functions for the divisor $\pi_1^{-1}(\Delta_{X, i})$. We conclude that the map $ev \circ \pi_1 \colon \pi_1^{-1}(N^*_X) \to \mathcal{D}(\Delta_X)$ is an evaluation map associated to the pair $(X_1, \Delta_1)$, and it can be identified with $ev$ itself via the isomorphisms $\pi_1^{-1}(N^*_X) \simeq N^*_X$ and $\mathcal{D}(\Delta_1)=\mathcal{D}(\Delta_X)$.
    \item[Step 2] We can suppose that there exists a birational morphism $f\colon (Y, \Delta_Y) \to (X, \Delta_X)$. 
     By construction, there exists a priori only a birational map $X \dashrightarrow Y$. Then choose $X_2$ a resolution of indeterminacy of the map which is an isomorphism over $M_{\mathrm{B}}$, i.e.\ \vspace{-0.3 cm}
    \[
    \begin{tikzcd}
    & X_2 \arrow["f"']{ld} \arrow{rd}&\\
    X \arrow[dashed]{rr} && Y.
    \end{tikzcd}
\]
Up to taking a dlt modification of $X_2$, we can suppose that $(X_2, \Delta_2 \coloneqq f^{-1}(\Delta_1)_{\mathrm{red}})$ is a $\Q$-factorial dlt compactification of $M_{\mathrm{B}}$. Therefore, \cref{lem:indep2} for $(X, \Delta_X)$ and $(Y, \Delta_Y)$ holds if and only if it it holds simultaneously for the pairs $(X, \Delta_X)$ and $(X_2, \Delta_2)$, and for the pairs $(X_2, \Delta_2)$ and $(Y, \Delta_Y)$. Hence, we can suppose that $Y$ dominates $X$, and we can now take $\mathfrak{s}\coloneqq r_{YX}$ as defined in \eqref{eq:defretraction}.

\item[Step 3] For any  $L \subseteq K$, let $\Delta_{X, J}$ be the minimum stratum of $\Delta_X$ such that $f(\Delta_{Y, L}) \subseteq \Delta_{X,J}$, with $J \subseteq I$. Assume first that $f(N_{Y, L})\subset N_{X,J}$. 
Consider the open cover of $N_Y$ 
\[U_i \coloneqq \bigcup_{k \colon f(N_{Y,k})\subset N_{X,i}} N_{Y,k}, \text{ with }i \in I.\]
By construction $\{\mathfrak{s} \circ ev_Y(i)\}_{i \in I}$ is a partition of unity subordinate to $\{U_i\}_{i \in I}$, while $\{ev_X (i) \circ f\}_{i \in I}$ is a partition of unity subordinate to $\{f^{-1}(N_i)\}_{i \in I}$. 
Note that for any $y \in U^{\circ}_J\coloneqq U_J \setminus \bigcup_{i \not \in J} U_i$, there exists $M \subseteq I$ containing $J$ such that $y \in f^{-1}((N_{X,M})^{\circ})$. Denoting $\sigma_J$ the (closed) cell in $\mathcal{D}(\Delta_X)$ corresponding to $\Delta_{X,J}$, we have that $\mathfrak{s} \circ ev_Y(y) \in \sigma_J \subseteq \sigma_M$ and $ev_X \circ f(y) \in \sigma_M$. Hence, the segment $t \, (\mathfrak{s} \circ ev_Y(y)) +(1-t)\, (ev_X \circ f(y))\subset \R^M$, with $t \in [0,1]$, is entirely contained in $\sigma_M$, and so the convex convolution of $\mathfrak{s} \circ ev_Y \circ  \Psi$ and $ev_X \circ \Psi$ gives the desired homotopy.

\item[Step 4] In Step 3 we assumed that $f(N_{Y,L})\subset N_{X,J}$ for any $L \subset K$. If this is not the case, we can shrink $N_{Y}= \beta^{-1}[0, \delta)$ to $N'_{Y}\coloneqq \beta^{-1}[0, \delta')$, with $\delta'\ll \delta$, such that $f(N'_{Y, L})\subset N_{X,J}$. Choose an auxiliary evaluation map $ev'_Y \colon N'_{Y} \to  \mathcal{D}(\Delta_Y)$. Applying Step 3 with $f=\mathrm{Id}_Y \colon Y \to Y$, we obtain that the maps $ev'_Y \circ \Psi$ and $ev_Y \circ \Psi$ are homotopic equivalent, and so again by Step 3 the following maps are homotopic to each other
\[\mathfrak{s} \circ ev_Y \circ \Psi \sim \mathfrak{s} \circ ev'_Y \circ \Psi \sim ev_X \circ \Psi.\]

\item[Step 5] By \cite[Theorem 28]{deFernexKollarXu2012} $\mathcal{D}(\Delta_X)$ and $\mathcal{D}(\Delta_Y)$ have the same homotopy type. If $\mathcal{D}(\Delta_X)$ has the homotopy of a sphere, then the retraction $r_{YX}$ has topological degree $\pm 1$, and so $\mathfrak{s}=r_{YX}$ is a homotopy equivalence by Hopf theorem.
\end{itemize}
\end{proof}

\section{Local computations on the Tate curve}\label{appendixA}

The goal of this appendix is to prove \cref{cor:appendix},
which is a technical ingredient needed in the proof of~\cref{lem:reduced lc logCY}. To this end, we recall the construction of the Tate curve. Following~\cite[{}VII]{DeligneRapoport},
it is a model over $\dS\coloneqq \C[[t]]$ of the multiplicative group $\mathbb{G}_{m}$ with special fibre given by an infinite chain of $\PP^1$'s.

\spa{
Let $(x_i,y_{i+1})_{i \in \Z}$ be a collection of indeterminates.
The Tate curve $\overline{\cG}_m$ over $\dS$ is the union of the affine charts $(U_{i+1/2})_{i \in \Z}$ given by
\[
U_{i+1/2}\coloneqq \Spec \left(\frac{\dS[x_i, y_{i+1}]}{(x_iy_{i+1}-t)} \right).
\]
For each $i \in \Z$, the charts $U_{i-1/2}$ and $U_{i+1/2}$ are glued along the open subscheme
\begin{alignat*}{2}
T_i \coloneqq U_{i-1/2} \cap U_{i+1/2} & = \Spec \left(\O(U_{i+1/2})[x_i^{-1}] \right)= \Spec \left(\dS[x_i, x^{-1}_i ]\right) \qquad \quad && (y_{i+1}= t/x_i)\\
& = \Spec \left(\O(U_{i-1/2})[y_i^{-1}]\right)= \Spec\left( \dS[y_i, y^{-1}_i ] \right)\qquad \quad && (x_{i-1}= t/y_i)
\end{alignat*}
via the identification $x_iy_i =1.$ 
}

\spa{
The $\dS$-group scheme $\cG_m \coloneqq \bigcup_{i \in \Z} T_i$, obtained from $\overline{\cG}_m$ by removing the nodes in the special fibre, is the N\'{e}ron model of the multiplicative group $\mathbb{G}_m$, as explained in~\cite[Example 1.2.c]{DeligneRapoport}. 
In particular, the $n$-th multiplication map
$
\mathbb{G}_m \times \ldots \times \mathbb{G}_m \to \mathbb{G}_m 
$
extends to a homomorphism
\[
\mu_n \colon \cG^n_m \coloneqq \cG_m \times_\dS \ldots \times_\dS \cG_m \longrightarrow \cG_m 
\]
of $\dS$-group schemes. When $n=2$, $\mu_n$ is
given in local charts by
\begin{alignat*}{2}
    T_i \times_{\dS} T_j & \longrightarrow && T_{i+j}\\
    \dS[x_i, x^{-1}_i]\otimes_{\dS} \dS[x_j, x^{-1}_j] & \longleftarrow && \dS[x_{i+j}, x^{-1}_{i+j}]\\
    x_i\otimes x_j & \: \mapsfrom && x_{i+j}.
\end{alignat*}
As $x_{i-1}y_i =t$ and $x_i y_i =1$,
it follows that
$x_i = t^{-i}x_0$. In particular, the identity section $\mathcal{I}d$ of $\cG_m$ is cut out in the chart $T_i$ by the equation $x_i = t^{-i}$.

Let $\cV_{n-1}\coloneqq \mu_n^{-1}(\mathcal{I}d)$ be the fibre of the identity section $\mathcal{I}d$ via $\mu_n$, and let $\overline{\cV}_{n-1}$ denote the closure of $\cV_{n-1}$ in $\overline{\cG}_m^n$. 
\cref{pairnormalreducedlc} describes some properties of $(\overline{\cV}_{n-1}, \overline{\cV}_{n-1,0})$, where $\overline{\cV}_{n-1,0}$ is the special fibre of $\overline{\cV}_{n-1}$;
\begin{equation*}
\arraycolsep=1.4pt
\begin{array}{ccccc}
     & & \cV_{n-1} & \subset & \cG_m^n \\
     & & \cap &  & \cap   \\
     & & \overline{\cV}_{n-1} & \subset &  \overline{\cG}_m^n. 
\end{array}
\end{equation*}
}

\begin{proposition}\label{pairnormalreducedlc}
The pair $(\overline{\cV}_{n-1}, \overline{\cV}_{n-1, 0})$ is normal, reduced, and toric, i.e. $\overline{\cV}_{n-1}$ is the formal completion of a normal toric scheme along its reduced toric boundary $\overline{\cV}_{n-1, 0}$.
Furthermore, the intersection $\overline{\cV}_{n-1} \setminus \cV_{n-1}$
has codimension two in $\overline{\cV}_{n-1}$.
\end{proposition}

The~\cref{pairnormalreducedlc} is an immediate corollary of~\cref{local computation} below. 
Indeed, the assertions in~\cref{pairnormalreducedlc} are local: we may work on the the open subsets
$$
U_{\alpha}\coloneqq U_{\alpha_1 + 1/2} \times_{\dS} \ldots \times_{\dS} U_{\alpha_n + 1/2},
$$
for any multi-index $\alpha=(\alpha_1, \ldots, \alpha_n) \in \Z^n$, since
the 
$U_{\alpha}$'s cover $\overline{\cG}_m^n$.
Let $\cV_{\alpha}$ be the restriction of $\cV_{n-1}$ to $T_{\alpha}\coloneqq T_{\alpha_1} \times_{\dS} \ldots \times_{\dS} T_{\alpha_n}$, $\overline{\cV}_{\alpha}$ be its closure in $U_{\alpha}$, and $\overline{\cV}_{\alpha, 0}$ be the special fibre of $\overline{\cV}_{\alpha}$. 
In local coordinates, 

\begin{minipage}{0.3\textwidth}
\begin{equation*}
\arraycolsep=1.4pt
\hspace{80pt}\begin{array}{ccccccc}
     & & \cV_{\alpha} & \subset & T_\alpha & & \\
     & & \cap &  & \cap &  &  \\
     & & \overline{\cV}_{\alpha} & \subset & U_\alpha & & 
\end{array}
\end{equation*}
\end{minipage}
\begin{minipage}{0.7\textwidth}
\begin{align*}
U_{\alpha}& = \Spec \left(\frac{\dS[x_{\alpha_1}, y_{\alpha_{1}+1}, \ldots, x_{\alpha_n}, y_{\alpha_{n}+1}]}{ x_{\alpha_1}y_{\alpha_1 +1}=\ldots =x_{\alpha_n}y_{\alpha_n +1}=t } \right)\\
T_{\alpha}& = \Spec \left( \dS[x_{\alpha_1}^{\pm 1}, \ldots, x_{\alpha_n}^{\pm 1}]\right)\subset U_{\alpha},\\
\cV_{\alpha}& =\left\{ \prod^n_{i=1} x_{\alpha_i} = t^{- \sum \alpha_i}\right\} \subset T_{\alpha}.
\end{align*}
\end{minipage}

\begin{lemma}\label{local computation}
For any $\alpha \in \Z^n$, the pair $(\overline{\cV}_{\alpha}, \overline{\cV}_{\alpha, 0})$ is normal, reduced, and toric. 
Furthermore, the intersection $\overline{\cV}_{\alpha} \setminus \cV_\alpha$ has codimension two in $\overline{\cV}_{\alpha}$.
\end{lemma}

\begin{proof}
The proof is divided into cases depending on the sign of $|\alpha | \coloneqq \sum_{i=1}^n \alpha_i$.
\begin{enumerate}

\item[Case 1.] Assume $|\alpha| >0$. The equation $t^{|\alpha|} \prod_{i=1}^n x_{\alpha_i} = 1$ vanishes on $\overline{\cV}_\alpha$. This implies that $t$ and all $x_{\alpha_i}$ are invertible on it, thus $\overline{\cV}_{\alpha,0} = \emptyset$ and $\overline{\cV}_\alpha= \cV_\alpha \simeq \mathbb{G}^{n-1}_{\C((t))}$.

\item[Case 2.] Assume $|\alpha|=0$. 
The equation $\prod_{i=1}^n x_{\alpha_i} = 1$ implies that all $x_{\alpha_i}$ are invertible on $\overline{\cV}_\alpha$.
We obtain that
$$
\overline{\cV}_{\alpha} \subseteq 
\{ \prod_{i=1}^n x_{\alpha_i} = 1, \, x_{\alpha_i}y_{\alpha_i +1} = t \} \simeq \mathbb{G}^{n-1}_R, 
$$
thus $(\overline{\cV}_{\alpha}, \overline{\cV}_{\alpha, 0})=(\mathbb{G}^{n-1}_R, \mathbb{G}^{n-1}_\C)$ and $\overline{\cV}_{\alpha} \setminus \cV_\alpha = \emptyset$.

\item[Case 3.] Assume $|\alpha|<0$. 
We will show that $\overline{\cV}_{\alpha}$ is normal by proving the conditions $\mathrm{S}_2$ and $\mathrm{R}_1$, 
and in the process we deduce that $(\overline{\cV}_{\alpha}, \overline{\cV}_{\alpha, 0})$ is toric, $\overline{\cV}_{\alpha, 0}$ is reduced and $\overline{\cV}_{\alpha}\setminus {\cV}_{\alpha}$ has codimension two in $\overline{\cV}_{\alpha}$.
\begin{enumerate}

\item[Step 1.] 
We prove that $\overline{\cV}_{\alpha}$ is Cohen-Macaulay (hence $\mathrm{S}_2$) by applying \cite[Lemma 7]{Kollar2011}: if a Gorenstein scheme of pure dimension $d$ is a union of two closed subschemes of pure dimension $d$ and one is Cohen-Macaulay, then the other is Cohen-Macaulay.

Let $\cZ_\alpha$ be the closed toric subscheme of $U_\alpha$ given by the equations
\[
\cZ_\alpha \coloneqq \begin{cases}
t \big( \prod^n_{i=1} x_{\alpha_i}-  t^{-|\alpha|} \big)=0,\\
x_{\alpha_1}y_{\alpha_1 +1}=\ldots =x_{\alpha_n}y_{\alpha_n +1}=t.
\end{cases}
\]
We have 
\begin{align*}
    &\cZ_{\alpha, 0}
     = U_{\alpha,0} = \{x_{\alpha_1}y_{\alpha_1 +1}=\ldots =x_{\alpha_n}y_{\alpha_n +1}=0 \}, \\
    &\cZ_\alpha \cap \{t \neq 0 \} 
     = \cV_\alpha \cap \{t \neq 0 \} \simeq \mathbb{G}_{m, \C((t))}^{n-1}, 
\end{align*}
so $\cZ_\alpha = U_{\alpha,0} \cup \overline{\cV}_\alpha$. We can then apply  \cite[Lemma 7]{Kollar2011} as $\cZ_\alpha$ and $U_{\alpha,0}$ are complete intersections. Hence, $\overline{\cV}_\alpha$ is Cohen-Macaulay.
In particular, the pair $(\overline{\cV}_{\alpha}, \overline{\cV}_{\alpha, 0})$ is toric, as both $\overline{\cV}_{\alpha}$ and $\overline{\cV}_{\alpha, 0}$
are torus-invariant subschemes of $\cZ_{\alpha}$ and $\overline{\cV}_{\alpha} \setminus \overline{\cV}_{\alpha,0} \simeq \mathbb{G}_{m,\C((t))}^{n-1}$.

\item[Step 2.] To prove that $\overline{\cV}_\alpha$ is  $\mathrm{R}_1$, it suffices to check that $\overline{\cV}_\alpha$ is smooth at the generic point of the irreducible components of $\overline{\cV}_{\alpha,0}$, since $\overline{\cV}_\alpha \cap \{t \neq 0 \}=\cV_\alpha \cap \{t \neq 0\}$ is smooth.
Let $D \subseteq \overline{\cV}_{\alpha,0}$ be one of them. Up to relabelling, there exists an integer $m$ such that the chart $U$ of $U_{\alpha,0}$
$$
U \coloneqq
\left\{ \begin{aligned}
& x_{\alpha_1}= \ldots = x_{\alpha_m} = y_{\alpha_{m+1} +1}= \ldots = y_{\alpha_{n} +1}=0  \\
& y_{\alpha_{2}+1} \neq 0, \ldots , y_{\alpha_{m}+1} \neq 0,  x_{\alpha_{m+1}}\neq 0, \ldots , x_{\alpha_{n}} \neq 0
\end{aligned}
\right\} \simeq \mathbb{G}_{m, \C}^{n-1} \times \mathbb{A}^1_\C
$$
contains the generic point of $D$.
At the generic point of $D$ we have
\begin{equation*}
\begin{cases}
x_{\alpha_i}= x_{\alpha_1} y_{\alpha_1 + 1} y_{\alpha_i+1}^{-1} & \textrm{ for }2 \leqslant i \leqslant m \\
y_{\alpha_i+1}= x_{\alpha_1} y_{\alpha_1 + 1} x_{\alpha_i}^{-1} & \textrm{ for }m+1 \leqslant i \leqslant n, \\
\end{cases}
\end{equation*}
and $\overline{\cV}_\alpha$ is an irreducible component of the locus given by the equation
\begin{equation}
\label{equ:W_alpha}
\prod_{i=1}^{n} x_{\alpha_i} 
= x_{\alpha_1}^{m} y_{\alpha_1 +1}^{m-1} \underbrace{\bigg( \prod_{j=2}^m y_{\alpha_{j}+1}^{-1}\prod_{i=m+1}^n x_{\alpha_i}\bigg)}_{u} 
= x_{\alpha_1}^{-|\alpha|} y_{\alpha_1 +1}^{-|\alpha|}.
\end{equation}

If $m > - |\alpha| $, then \cref{equ:W_alpha} implies that $x_{\alpha_{1}}^{-|\alpha|} y_{\alpha_1 + 1}^{{-|\alpha|}}(u x_{\alpha_{1}}^{m+|\alpha|} y_{\alpha_1 + 1}^{{m-1+|\alpha|}}- 1)=0$. This is impossible: $\overline{\cV}_\alpha$ would be cut by the equation $u x_{\alpha_{1}}^{m+|\alpha|} y_{\alpha_1 + 1}^{{m-1+|\alpha|}}= 1$ at the generic point of $D$, but then $D \not\subseteq \overline{\cV}_\alpha$.

\noindent The same argument proves that the case  $m < - |\alpha| $ cannot occur.

If $m = - |\alpha|$, then \cref{equ:W_alpha} implies that $x_{\alpha_{1}}^m y_{\alpha_1 + 1}^{m-1}(u- y_{\alpha_1 +1})=0$. It follows that 
\begin{equation}\label{eq:R1}
    \overline{\cV}_\alpha 
\stackrel{\text{loc at $D$}}{=} \{u- y_{\alpha_1 +1} =0 \}, 
\end{equation}
hence $\overline{\cV}_\alpha$ is $\mathrm{R}_1$, and $\overline{\cV}_{\alpha, 0}
\stackrel{\text{loc at $D$}}{=} \{x_{\alpha_1} =0 \}$ is reduced.
\item[Step 3] A point in $\overline{\cV}_{\alpha} \setminus \cV_\alpha \subseteq U_{\alpha} \setminus T_\alpha$ is characterised by the property that a pair of coordinates $(x_{\alpha_i}, y_{\alpha_i +1})$ vanishes simultaneously. Hence, \eqref{eq:R1} yields that $\cV_\alpha$ coincides with $\overline{\cV}_{\alpha}$ at the generic point of $D$. Therefore, 
 $\overline{\cV}_{\alpha} \setminus \cV_\alpha$ has codimension two in $\overline{\cV}_{\alpha}$.
\end{enumerate}
\end{enumerate}
\end{proof}

Let $\cX_{n-1}$ and $\overline{\cX}_{n-1}$ be as in \cref{construct dlt compact} and consider the diagram \cref{equ: diagram Ksing}. Given the uniformisation $p\colon  \overline{\cG}_m \rightarrow \mathscr{E}$, we have
\begin{equation*}
\begin{tikzcd}
\overline{\cV}_{n-1} \times_R \overline{\cV}_{n-1} 
\arrow[hookrightarrow, r]  
\arrow[d]  
& (\overline{\cG}_m \times_R \overline{\cG}_m)^n
\arrow[d,  "{P \coloneqq (p,p)^n}"] \\
\overline{\cX}_{n-1}
\arrow[hookrightarrow, r] 
& (\mathscr{E} \times_R \mathscr{E})^n
\end{tikzcd}
\end{equation*}

\begin{corollary} \label{cor:appendix}
The pair $(\overline{\cX}_{n-1}, \overline{\cX}_{n-1,0})$ is reduced lc and $\overline{\cX}_{n-1} \setminus \cX_{n-1}$ has codimension two in $\overline{\cX}_{n-1}$.
\end{corollary}

\begin{proof}
The restriction of $P$ to $\overline{\cV}_{n-1} \times_R \overline{\cV}_{n-1}$ is \'{e}tale. Hence it suffices to prove the statement for the pair $\big(\overline{\cV}_{n-1} \times_R \overline{\cV}_{n-1} , (\overline{\cV}_{n-1} \times_R \overline{\cV}_{n-1})\,_0 \big)$. This follows directly from \cref{pairnormalreducedlc}, taking the fibre product.
\end{proof}

\printbibliography
\end{document}